\documentclass[10pt,letterpaper,reqno]{amsart}
\usepackage{color}
\usepackage[latin1]{inputenc}
\usepackage{amsmath}
\usepackage{amsfonts}
\usepackage{amsthm}
\usepackage{epsfig}
\usepackage{graphicx}
\usepackage{amssymb}
\usepackage{bm}
\usepackage{esint}
\usepackage{caption}

\newcommand{\B}{\mathcal B}

\renewcommand{\H}{\mathcal{H}}
\renewcommand{\L}{\mathcal{L}}

\newcommand{\X}{\mathcal{X}}

\newcommand{\R}{\mathbb{R}}

\newcommand{\NN}{\mathcal{N}}

\newcommand{\Om}{\Omega}
\renewcommand{\S}{\Sigma}
\renewcommand{\a}{\alpha}

\newcommand{\g}{\gamma}
\newcommand{\de}{\delta}
\newcommand{\e}{\varepsilon}
\renewcommand{\k}{\kappa}
\renewcommand{\l}{\lambda}
\newcommand{\s}{\sigma}

\newcommand{\om}{\omega}
\newcommand{\vphi}{\varphi}

\newcommand{\Div}{{\rm div}\,}

\newcommand{\dist}{{\rm dist}}

\newcommand{\spt}{{\rm spt}}

\newcommand{\vol}{{\rm vol}\,}

\newcommand{\ov}{\overline}

\newcommand{\pa}{\partial}

\newcommand{\cc}{\subset\subset}

\DeclareMathOperator{\reach}{reach}
\DeclareMathOperator{\Unp}{{\rm UP}}

\newcommand{\Der}{\mathrm{D}}
\newcommand{\der}{\mathrm{d}}

\theoremstyle{plain}
\newtheorem{theorem}{Theorem}[section]
\newtheorem{lemma}[theorem]{Lemma}

\newtheorem*{theorem*}{Theorem}
\newtheorem*{corollary*}{Corollary}

\theoremstyle{definition}

\newtheorem{remark}{Remark}[section]

\newtheorem*{notation*}{Notation}

\numberwithin{equation}{section}
\numberwithin{figure}{section}
\newcommand{\id}{{\rm id}\,}

\usepackage{xcolor}

\newcommand{\cred}{\color{red}}

\title{Rigidity and compactness \\ with constant mean curvature \\ in warped product manifolds}

\author{Francesco Maggi}
\address{Department of Mathematics, The University of Texas at Austin, 2515 Speedway, Stop C1200, Austin TX 78712-1202, United States of America}
\email{maggi@math.utexas.edu}
\author{Mario Santilli}
\address{Dipartimento di Ingegneria e Scienze dell'Informazione e Matematica, Universit\'a degli Studi dell'Aquila, 67100 L'Aquila, Italy}
\email{mario.santilli@univaq.it}

\begin{document}

\begin{abstract}
{\rm We prove the rigidity of {\it rectifiable} boundaries with constant {\it distributional} mean curvature in the Brendle class of warped product manifolds (which includes important models in General Relativity, like the deSitter--Schwarzschild and Reissner--Nordstrom manifolds). As a corollary we characterize limits of rectifiable boundaries whose mean curvatures converge, as distributions, to a constant. The latter result is new, and requires the full strength of distributional CMC-rigidity, even when one considers smooth boundaries whose mean curvature oscillations vanish in arbitrarily strong $C^k$-norms. Our method also establishes that rectifiable boundaries of sets of finite perimeter in the hyperbolic space with constant distributional mean curvature are finite unions of possibly mutually tangent geodesic spheres.} 
\end{abstract}


\maketitle

\tableofcontents

\vspace{-1cm}

\section{Introduction} \subsection{Overview} We move from two recent extensions of the classical Alexandrov theorem \cite{alexandrov}: {\it in the Euclidean space, spheres are the only connected, constant mean curvature (CMC) boundaries enclosing finite volumes}:

\medskip

\noindent (i) In \cite{brendle}, Brendle has proved CMC-rigidity in a class of warped product Riemannian manifolds which includes important models in General Relativity, like the deSitter--Schwarzschild and the Reissner--Nordstrom manifolds. In dimension $3\le n\le 7$, when isoperimetric sets are smooth by local regularity theorems, Brendle's theorem allows one to solve the ``horizon homologous'' isoperimetric problem in this class of warped product manifolds. In turn, since the works of Eichmair and Metzger \cite{EichmairMetzgerJDG12,EichmairMetzgerJDG13,EichmairMetzgerINV}, the study of horizon-homologous isoperimetric problems in the large volume regime has played a prominent role in the analysis of the Huisken--Yau problem \cite{huisken_yau}
-- see, e.g. \cite{chodoshCMP,CarlottoChodoshEichmair,chodoshEichmairVolkmanJDG,chodoshEichmairShiZhu,chodoshEichmairShiYu,ChodoshEichmairDuke}.

\medskip

\noindent (ii) In \cite{delgadinomaggi}, Delgadino and the first-named author have extended the Alexandrov theorem to the class of (Borel) sets with finite volume, finite {\it distributional} perimeter, and constant {\it distributional} mean curvature. CMC-rigidity in the distributional setting in turn implies, by basic varifolds theory, a general {\it compactness theorem for almost CMC-boundaries}: in the Euclidean space, finite unions of disjoint spheres with equal radii are the only possible limits of sequences of boundaries converging in volume and in perimeter, and whose mean curvatures converge, as distributions, to a constant. With stronger controls on the mean curvature oscillation one can even provide quantitative rates of convergence towards finite unions of balls, as done, for example, in \cite{ciraolomaggi2017,delgadinomaggimihailaneumayer,julin-nini,julin-morini-ponsiglione-spadaro}. In addition to their geometric interest, these results find applications to the resolution of geometric variational problems (see, e.g., \cite{ciraolomaggi2017,delgadinoWeser} for the characterization of {\it local} minimizers of ``free droplet energies'') and to the long time behavior of the volume-preserving mean curvature flow (see, e.g., \cite{julin-nini,morini-ponsiglione-spadaro,julin-morini-ponsiglione-spadaro}).

\medskip
\vspace{0.2cm}

The main results of this paper are the distributional version of the Alexandrov theorem in the Brendle class (Theorem \ref{thm main}) and a consequent compactness theorem for almost-CMC boundaries (Theorem \ref{thm compactness 1}).

\medskip
\vspace{0.2cm}

An interesting aspect of the method of proof of Theorem \ref{thm main} is that we avoid the use of ``smoothness intensive'' tools, like the Sch\"atzle {\it strong} maximum principle for integer varifolds \cite{schatzle}, which was central in \cite{delgadinomaggi}; or the Greene--Wu ``approximation by convolution'' theorem \cite{GW1,GW2}, which was crucially used in \cite{brendle}. Our method indeed employs the {\it metric} notion of set of positive reach in combination with one-sided {\it viscous} mean-curvature/dimensions bounds (as formulated by White in \cite{whiteMC}; see also \cite{caffarellicordoba}); and, as a by-product of these efforts, throughout our analysis, we only need to use the {\it weak} maximum principle (meant as ``one-sided inclusion and contact imply mean curvature ordering'').  This synthetic approach was originally developed by the second author in \cite{Santilli2019heintze-karcher} to prove the Heintze-Karcher inequality for sets of finite perimeter in Euclidean space (see also \cite[Theorem 4.2]{Santilli24}), and later implemented in the anisotropic setting in \cite{deRosaKolaSantilli}. We take the success of this synthetic approach as a significant indication that {\it the Alexandrov theorem should hold in some metric version of the Brendle class}. 

\subsection{The Brendle class and the main theorem} Given $n\ge 3$, we denote by $\B_n$ the class of the $n$-dimensional manifolds $(M,g)$ of the form
\begin{equation}
  \label{warped product}
  M=N\times[0,\bar{r})\,,\qquad g=dr\otimes dr+h(r)^2\,g_N\,,
\end{equation}
for some $\bar{r}\in(0,\infty]$,  compact $(n-1)$-dimensional Riemannian manifold $(N,g_N)$, and smooth positive function $h:[0,\bar{r})\to\R$, such that:
\begin{enumerate}
  \item[(H0)] for some $\rho \in \mathbb{R}$, ${\rm Ric}_N\ge\rho(n-2)\,g_N$ on $N$;
  \item[(H1)] $h'(0)=0$ and $h''(0)>0$;
  \item[(H2)] $h'>0$ on $(0,\bar{r})$;
  \item[(H3)] $2\,(h''/h)+(n-2)\,[((h')^2-\rho)/h^2]$ is increasing on $(0,\bar{r})$\,;
  \item[(H4)] $(h''/h)+[(\rho-(h')^2)/h^2]$ is positive on $(0,\bar{r})$\,.
\end{enumerate}
To obtain geometric interpretations of these conditions we denote by
\[
M^\circ=N\times(0,\bar{r})\,,\qquad  N_0=N\times\{0\}\,,\qquad N_t=N\times\{t\}\quad (t>0)\,,
\]
the {\bf interior}, the {\bf horizon}, and the {\bf slices} of $M$, and notice that
\begin{eqnarray}
\label{Ricci h}
{\rm Ric}_M\!\!\!\!&=&\!\!\!{\rm Ric}_N-\big\{h\,h''+(n-2)\,(h')^2\big\}\,g_N-(n-1)\,\big(h''\big/h\big)\,dr\otimes dr\,,\,\,\,\,
\\
\label{scalar h}
{\rm R}_M\!\!\!\!&=&\!\!\!\big({\rm R}_N\big/h^2\big)-(n-1)\,\Big(2\,\big(h''\big/h\big)+(n-2)\,\big(h'\big/h\big)^2\Big)\,.
\end{eqnarray}
The scalar mean curvature of $N_t$ with respect to its $g$-unit normal vector field $\pa/\pa r$ is $H_{N_t}=\langle\vec{H}_{N_t},\pa/\pa r\rangle_g=(n-1)\,h'(t)\big/h(t)$, so that the horizon of $M$ is a minimal surface by ${\rm (H1)}$, and the slices of $M$ have {\it positive} CMC (w.r.t. to $\pa/\pa r$) thanks to ${\rm (H2)}$. We next notice that ${\rm (H3)}$ implies (in combination with \eqref{scalar h}, ${\rm (H0)}$, and ${\rm (H2)}$) that ${\rm R}_M$ is {\it decreasing} along $\pa/\pa r$. Finally, while \eqref{Ricci h} implies that $\pa/\pa r$ is an eigenvector of ${\rm Ric}_M$, ${\rm (H4)}$ (combined with ${\rm (H0)}$) adds the information that $\pa/\pa r$ is a {\it simple} eigenvector. We then have:

\medskip

\noindent {\bf Brendle's theorem}: \cite[Theorem 1.1]{brendle} {\it If $n\ge 3$, $(M,g)\in\B_n$, and $\Sigma$ is a smooth closed, embedded, orientable, CMC hypersurface in $M$, then $\Sigma$ is  a slice of $M$.}

\begin{remark}[Dropping ${\rm (H4)}$]
  {\rm A simple remark (which seems to have gone uncommented so far) is that Brendle's theorem also holds in the class $\B_n^*$ of those $(M,g)$ satisfying \eqref{warped product}, ${\rm (H0)}$, ${\rm (H1)}$, ${\rm (H2)}$, and
  \medskip
  \begin{enumerate}
  \item[(H3)$^*$] $2\,(h''/h)+(n-2)\,[((h')^2-\rho)/h^2]$ is {\it strictly} increasing on $(0,\bar{r})$\,.
  \end{enumerate}
  \medskip
  In other words, condition ${\rm (H4)}$ is not needed to conclude rigidity as soon as ${\rm R}_M$ is {\it strictly} decreasing along $\pa/\pa r$. (For more details on this point, see the discussions in Section \ref{section organization} and Remark \ref{remark on the relation} below.) In terms of applications to General Relativity, it is interesting to notice that while the Reissner--Nordstrom manifolds belong to the class $\B_n^*$, the deSitter--Schwarzschild manifolds do not; in particular, a stronger stability mechanism for almost-CMC hypersurfaces is at work in the former class than in the latter; see Appendix \ref{appendix h3star} for more information.}
\end{remark}

\begin{remark}[Formulation with boundaries]
  {\rm Brendle's theorem is, actually, a statement about {\it boundaries} in $M$. Indeed, as noticed also in \cite[Section 3]{brendle}, under the assumptions of Brendle's theorem on $\Sigma$,
  \begin{equation}\label{from Sigma to Omega}\tag{$\Sigma\leftrightarrow\Omega$}
  \begin{split}
  &\mbox{there is $(a,b)\cc(0,\bar{r})$ and $\Om\subset M$ open such that}
  \\
  &\mbox{$\Sigma\subset N\times(a,b)$ and either $\pa\Om=\Sigma$ or $\pa\Om=\Sigma\cup N_0$}\,.
  \end{split}
  \end{equation}}
\end{remark}

Now, as explained in more detail later on, there are two basic geometric problems -- the characterization of horizon-homologous isoperimetric regions and the study of sequences of (smooth) boundaries with vanishing mean curvature oscillation -- that call for the extension of Brendle's theorem to the class of {\it sets of finite perimeter}. This extension is the content of our main theorem, Theorem \ref{thm main} below. Referring to \cite{maggiBOOK} for a complete discussion of the subject, we just recall here that a Borel set $\Om$ in $(M,g)$ is a {\it set of finite perimeter} if ${\rm Per}(\Om):=\sup\{\small\int_\Om\Div X\,d\H^n:X\in\X(M)\,,|X|_g\le1\}$ is finite (where $\mathcal{X}(M)=\{\mbox{smooth vector fields on $M$}\}$). Then one can define the {\it reduced boundary} $\pa^*\Om$ (a locally $\H^{n-1}$-rectifiable set with ${\rm Per}(\Om)=\H^{n-1}(\pa^*\Om)$) and the measure-theoretic outer normal $\nu_\Om$ (a Borel $g$-unit vector field defined on $\pa^*\Om$) so that the distributional divergence theorem  $\small{\int}_\Om\,\Div\,X\,d\H^n=\small{\int}_{\pa^*\Om}\langle X,\nu_\Om\rangle_g\,d\H^{n-1}$ holds for every $X\in\mathcal{X}(M)$. Finally, one says that $H$ is the {\it distributional mean curvature} of $\pa^*\Om$ with respect to $\nu_\Om$, if $H$ is summable in $\H^{n-1}\llcorner\pa^*\Om$, and
\[
\int_{\pa^*\Om}\,\Div^{\pa^*\Om} X\,d\H^{n-1}=\int_{\pa^*\Om} H\,\langle X,\nu_\Om\rangle_g\,d\H^{n-1}\,,\qquad\forall X\in\mathcal{X}(M)\,,
 \]
where $\Div^{\pa^*\Om}X:=\Div X-\langle\nabla_{\nu_\Om}X,\nu_\Om\rangle_g$. We thus have the following {\it distributional} version of Brendle's theorem:

\begin{theorem}[Rigidity]\label{thm main}
  If $n\ge 3$, $(M,g)\in\B_n\cup\B_n^*$, $\Om$ is a set of finite perimeter in $M$ such that $ M^\circ \cap  \overline{\pa^\ast \Omega}$ is compact in $M$ and $\l\in\R$ is such that
  \begin{equation}
    \label{distributional CMC}
      \int_{M^\circ\cap\pa^*\Om}\!\!\Div^{\pa^*\Om}X\,d\H^{n-1}=\l\,\int_{M^\circ\cap\pa^*\Om} \langle X,\nu_\Om\rangle_g\,d\H^{n-1}\,,
  \end{equation}
  for every $X\in\X(M)$, then, for some $t_0\in(0,\bar{r})$, either $\Om$ is $\H^n$-equivalent to $N\times(0,t_0)$ (and $\l>0$) or $\Om$ is $\H^n$-equivalent to $N\times(t_0,\bar{r})$ (and $\l<0$).
\end{theorem}

\begin{remark}
  {\rm By Allard's regularity theorem \cite{Allard,SimonLN,camilloINVITE}, \eqref{distributional CMC} implies that $\Sigma=\pa^*\Om$ is a smooth, CMC hypersurface in $M$ with $\H^{n-1}(\ov\Sigma\setminus\Sigma)=0$. Since $\Sigma$ {\it is not necessarily closed}, rigidity cannot be deduced by the direct application of Brendle's theorem. The difficulty addressed in Theorem \ref{thm main} is expressing how the distributional CMC condition \eqref{distributional CMC} ``ties together'' the (potentially countably many) connected components of $\Sigma$, and forces them to align into a single slice, rather than, say, allowing them to combine through the singular set $\ov\Sigma\setminus\Sigma$ into a non-slice CMC hypersurface.}
\end{remark}


\begin{remark}
  {\rm Condition \eqref{distributional CMC} is equivalent to ask that, if $f_t$ is a smooth volume-preserving flow of $\Om$ ($f_0=\id$, $(\pa f/\pa t)|_{t=0}=X\in\X(M)$, and $\H^n(f_t(\Om))=\H^n(\Om)$ for every $|t|$ small), then
  \[
  \big(d\big/dt\big)\big|_{t=0}\,{\rm Per}(f_t(\Om))=0\,.
  \]
  Theorem \ref{thm main} then says that {\it among sets of finite perimeter, slices are the only volume-constrained critical points of the area functional in $(M,g)\in\B_n\cup\B_n^*$}.}
\end{remark}

\begin{remark}[Necessity of ``$M^\circ\cap\ov{\pa^*\Om}$ compact'' in Theorem \ref{thm main}]
  {\rm If $v>0$ is small enough and $\Om_v$ is a minimizer of $\H^{n-1}(M^\circ\cap\pa^*\Om)$ among sets $\Om\subset M$ with $\H^n(\Om)=v$, then $M^\circ\cap\ov{\pa^*\Om_v}$ is a smooth CMC hypersurface, diffeomorphic to a hemisphere sitting on the horizon (see, e.g., \cite{maggimihaila} for a detailed analysis of this kind of result in the capillarity setting). Alternatively, one can first apply the perturbative construction of Pacard and Xu \cite{PacardXu} on the horizon of the {\it doubled} Schwarzschild manifold, as described in \cite{BrEichJDG13}. Either way, one obtains non-slice, CMC hypersurfaces bounding sets $\Om$.}
\end{remark}

\subsection{Compactness for almost-CMC boundaries} Theorem \ref{thm main} is of course strongly motivated by the following natural {\bf compactness problem for almost-CMC hypersurfaces}, which can be formulated, in very general terms, as follows:
\begin{equation}\tag{{\rm CP}}
  \label{CP}
  \begin{split}
    &\mbox{{\it If $(M,g)\in\mathcal{B}_n\cup\B_n^*$, does every sequence $\{\S_j\}_j$ of closed,}}
    \\
    &\mbox{{\it embedded, orientable smooth hypersurfaces in $M$}}
    \\
    &\mbox{{\it whose scalar mean curvatures in $(M,g)$ converge}}
    \\
    &\mbox{{\it to a constant, have slices as their only possible limits?}}
  \end{split}
\end{equation}
This basic question appears naturally in several contexts. Two important examples are: the analysis of the Huisken--Yau problem \cite{huisken_yau,qingtian2007,NevesTian1,NevesTian,Huang,lammMetzgerSchulze,BrendleEichINV2014,nerzJDG18,chodoshEichmairCrelle}, e.g., an outlying CMC hypersurface in an asymptotically Schwarzschild manifold can be seen as an almost-CMC hypersurface in the Schwarzschild manifold; and the study of the long-time behavior of the volume-preserving mean curvature flow -- since the $L^2$-oscillation of the mean curvature is the dissipation of the flow.

\medskip

In light of \eqref{from Sigma to Omega}, for each $\Sigma_j$ in \eqref{CP} there are open sets $\Om_j$ in $M$ and intervals $(a_j,b_j)\cc(0,\bar{r})$ such that $\Sigma_j\subset N\times(a_j,b_j)$ and (up to extracting subsequences) either $\Sigma_j=\pa\Om_j$ (for every $j$) or $\Sigma_j\cup N_0=\pa\Om_j$ (for every $j$). In both cases, under the natural set of assumptions that
\begin{eqnarray}
  \mbox{$a=\inf_j\,a_j>0$}\,,\qquad\mbox{$b=\sup_j b_j<\bar{r}$}\,,\qquad\mbox{$\sup_j\H^{n-1}(\Sigma_j)<\infty$}\,,
\end{eqnarray}
one finds a set $\Om$ with finite perimeter in $(M,g)$ such that, and up to extracting subsequences, it holds
\begin{equation}
  \label{Omegaj to Omega}
  \lim_{j\to\infty}\H^n(\Om_j\Delta\Om)=0\,,\qquad\liminf_{j\to\infty}\H^{n-1}(\pa\Om_j)\ge{\rm Per}(\Om)\,.
\end{equation}
In the above basic setting, Theorem \ref{thm main}, combined with standard closure results for integer varifolds, leads to an affirmative answer to \eqref{CP}. This is the content of our second main result, where the most general case of sequences of sets of finite perimeter is directly addressed:

\begin{theorem}[Compactness]
  \label{thm compactness 1}
  If $n\ge 3$, $(M,g)\in\B_n\cup \B_n^*$, and $\{\Om_j\}_j$ is a sequence of sets of finite perimeter in $M$ such that
  \begin{enumerate}
    \item[(i)] there are $(a,b)\cc(0,\bar{r})$ and a Borel set $\Om\subset M$ s.t. $M^\circ\cap\ov{\pa^*\Om_j}\subset N\times(a,b)$ for every $j$, and $\H^n(\Om_j\Delta\Om)\to 0$ as $j\to\infty$;
    \item[(ii)] ${\rm Per}(\Om_j)\to{\rm Per}(\Om)$ as $j\to\infty$;
    \item[(iii)] there is $\l\in\R$ such that, for every $X\in\X(M)$, as $j\to\infty$,
        \begin{equation}\label{very weak}
        \int_{M^\circ\cap\pa^*\Om_j}\Div^{\pa^*\Om_j}X\,d\H^{n-1}-\l\,\int_{M^\circ\cap\pa^*\Om_j}\langle X,\nu_{\Om_j}\rangle_g\,d\H^{n-1}\to 0\,;
        \end{equation}
  \end{enumerate}
   then there is $t_0\in(0,\bar{r})$ such that $\Om$ if $\H^n$-equivalent either to $N\times(0,t_0)$ (and then $\l>0$) or to $N\times(t_0,\bar{r})$ (and then $\l<0$).
\end{theorem}

\begin{remark}
  {\rm The proof of Theorem \ref{thm compactness 1} requires the full strength of Theorem \ref{thm main} even if one is only interested in sequences $\{\S_j\}_j$ of closed, embedded, orientable {\it smooth} hypersurfaces in $M$ whose mean curvature oscillations are assumed to vanish in every $C^k$-norm.}
\end{remark}

\begin{remark}
  {\rm In \eqref{very weak} the mean curvature oscillation is required to vanish {\it only in distributional sense}. This feature points to the possibility of applying Theorem \ref{thm compactness 1} to minimizing sequences of horizon-homologous isoperimetric problems that have been suitably {\it selected} by means of the Ekeland variational principle; see, e.g., \cite[Theorem 3.2(iv)]{CicaleseLeonardi}. Similarly, Theorem \ref{thm compactness 1} will be easily applied to the study of horizon-homologous isoperimetric sets with fixed volume with respect to metrics $\{g_j\}_j$ on $M$ such that, as $j\to\infty$,  $g_j\to g$ with $(M,g)\in\B_n\cup\B_n^*$. In both these examples, the perimeter convergence assumption (ii) is trivially checked by energy comparison.}
\end{remark}

\subsection{Strategy of proof and organization of the paper}\label{section organization} Section \ref{section positive reach} and Section \ref{section viscous mc} are devoted to establishing in the Riemannian setting the several tools from GMT that lie at the core of our analysis. In Section \ref{section normal bundles} we collect several properties of normal bundles to closed sets in complete Riemannian manifolds (Theorem \ref{thm normal bundles closed sets}), extending from the Euclidean case a series of recent results obtained in \cite{albano,mennesantilli19,santilliAMPA20,kolasinski2021regularity,HugSantilli2022}. In Section \ref{subsection positive reach} we review some theorems of Kleinjohann \cite{Kleinjohann} and Bangert \cite{Bangert} concerning sets of positive reach in complete Riemannian manifolds (Theorem \ref{thm kleinjohann} and Theorem \ref{thm bangert}). In Section \ref{section white} we recall the viscous notion of ``being $m$-dimensional with mean curvature vector bounded by $\lambda$'' introduced on closed subsets of Riemannian manifolds by White in \cite{whiteMC}, and recall its relation to distributional mean curvature (Theorem \ref{thm white}). Finally, in Section \ref{section mario}, we extend to the Riemannian setting a delicate ``Lusin-type property'' of normal bundles of White's $(m,\l)$-sets that, in the Euclidean case, was proved in \cite{SantilliBMS} (Theorem \ref{thm Santilli riem}).

\medskip

At the basis of the distributional version of the Alexandrov theorem proved in \cite{delgadinomaggi} (as well as of the previously cited quantitative versions of it), lies the approach to CMC-rigidity of Ros \cite{ros87ibero} and Montiel-Ros \cite{montielros}. Their method is based on the analysis of equality cases in the (Euclidean) Heintze-Karcher inequality.  Brendle's theorem, in turn, is based on the analysis of the equality cases of {\it two} different Heintze--Karcher-type inequalities for subsets $\Om$ of $(M,g)\in\B_n\cup\B_n^*$ (corresponding to the cases $\pa\Om=\Sigma$ and $\pa\Om=\Sigma\cup N_0$ appearing in \eqref{from Sigma to Omega}), and, specifically, to the fact that, when $\Sigma$ is such an equality case, then, $\Sigma$ is umbilic; the umbilicality of $\Sigma$ is then combined with ${\rm (H4)}$ to deduce rigidity (i.e., $\Sigma$ is a slice).

\medskip

A natural strategy for proving Theorem \ref{thm main} thus consists in: (a) establishing the two Heintze--Karcher-type inequalities of Brendle on sets with finite perimeter; (b) addressing the analysis of their equality cases in the distributional setting; and, (c) deducing rigidity from an established set of necessary conditions for equality.

\medskip

In Section \ref{section HK} we implement this strategy. There we work with {\it not necessarily closed}, smooth, embedded, hypersurfaces $\Sigma$ satisfying three assumptions: first, $\H^{n-1}(\ov\Sigma\setminus\Sigma)=0$ with $\H^{n-1}(\Sigma)<\infty$ ($\ov\Sigma\setminus\Sigma$ is understood as the ``singular set'' of $\Sigma$); second, either $\ov\Sigma=\pa\Om$, or $N_0\cup\ov\Sigma=\pa\Om$, for an open set $\Om$ in $M$; and, finally, $\ov\Sigma$ is compactly contained in $M^\circ$ and has bounded mean curvature in the viscous sense of White, while $\Sigma$ has positive mean curvature with respect to the outer $g$-unit normal $\nu_\Om$ to $\Om$.

\medskip

In Theorem \ref{thm hk inq} we take care of steps (a) and (b). Implementing step (a) is particularly delicate not only because, as expected, several passages of Brendle's argument make a crucial use of smoothness, and thus require considerable effort to be repeated or redesigned in a non-smooth framework;  but also because, in anticipation of the non-smooth rigidity discussion of step (c), we need a more detailed list of necessary conditions for equality cases in Brendle's Heintze--Karcher-type inequalities. In this direction, we notice that we shall establish {\it three} such conditions: (E1) $\Sigma$ is umbilical in $(M,g)$ (which is the condition already pointed out in \cite{brendle}); (E2) $M\setminus\Om$ has positive reach in $(M,g^*)$ (where $g^*$ is a certain complete metric on $M$, conformal to $g$); (E3) under ${\rm (H3)^*}$, $g_N(\nu_\Om,\nu_\Om)=0$ on $\Sigma$.  The identification of condition (E3) is new even in the smooth case. We also notice that, in the smooth case, condition (E2) is automatically true; and, indeed, the derivation of (E2) requires some careful work on Brendle's argument which is specific to the viscous setting.

\medskip

In Theorem \ref{thm rigidity}, we further assume that $\ov\Sigma$ has constant mean curvature in distributional sense, and then address step (c). If ${\rm (H3)^*}$, and thus condition (E3), holds, then we can infer rigidity directly from it, without using umbilicality: indeed (E3) implies the very strong information that $\nu_\Om$ is parallel to $\pa/\pa r$  $\H^{n-1}$-a.e. along $\Sigma$ -- an information that gives $\Sigma=N_{t_0}$ by a simple property of sets of finite perimeter (cf. with \cite[Exercise 15.18]{maggiBOOK}). If, otherwise, only ${\rm (H3)}$ holds, then, as in \cite{brendle}, we need to combine ${\rm (H4)}$ and umbilicality to deduce that, $\H^{n-1}$-a.e. on $\Sigma$, $\nu_\Om$ is either parallel {\it or orthogonal} to $\pa/\pa r$. The dichotomy parallel/orthogonal prevents the use of something as simple as \cite[Exercise 15.18]{maggiBOOK}. In the smooth case, one immediately excludes ``orthogonality'', and thus conclude rigidity, by a sliding argument. However, in the non-smooth setting, sliding arguments are not equally effective (because of multiplicity issues preventing the use of Allard's regularity theorem at contact points). This is the passage where condition (E2) reveals useful, and ultimately allows us to conclude the proof of Theorem \ref{thm rigidity}.

\medskip

Finally, Section \ref{section proofs} contains the proofs of Theorem \ref{thm main} and Theorem \ref{thm compactness 1}.

\subsection{Rigidity and bubbling in the hyperbolic space} 
It is a classical and well known result that Alexandrov rigidity result can be generalized to \emph{smooth} boundaries in the hyperbolic space. In \cite[Theorem 1.4]{brendle} a new proof is obtained based on the Heintze-Karcher inequality. It is a natural and interesting question to understand if finite unions of possibly mutually tangent balls are the only examples of sets of finite perimeter in the hyperbolic space with constant distributional mean curvature. The methods developed in this paper provides a positive answer to this question. 

\begin{theorem}\label{thm main hyperbolic}
	Suppose $ \Omega $ is a set of finite perimeter with compact closure in the hyperbolic space $ \mathbb{H}^n $ and $ \lambda \in \mathbb{R} $ such that 
	\begin{equation}\label{thm main hyperbolic eq}
		\int_{\partial^\ast \Omega} {\rm div}^{\partial^\ast \Omega}X\, d\mathcal{H}^{n-1} = \lambda \int_{\partial^\ast \Omega} \langle X, \nu_\Omega \rangle_{\mathbb{H}^n}\, d\mathcal{H}^{n-1},
	\end{equation}  
	for every $ X \in \X(M) $. Then $ \Omega $ is a finite union of disjoint (possibly mutually tangent) open geodesic balls with equal radii.
\end{theorem}

\begin{remark}\label{rmk: tame varifolds}
	In \cite{haslhoferOrWhite} the authors develop a beautiful moving plane method for a class of varifolds satisfying a suitable \emph{tameness condition} and they employ it to prove several rigidity results for stationary and CMC varifolds. In particular, cf.\ \cite[Theorem 1.8]{haslhoferOrWhite}, they prove that if $ \Sigma \subseteq \mathbb{H}^{n} $ is the support of a $ (n-1) $-dimensional \emph{tame} varifold without boundary and with constant mean curvature\footnote{cf.\ \cite[eq.\ (8)]{haslhoferOrWhite} for the definition of varifold with bounded mean curvature, and \cite[Definition 1.6]{haslhoferOrWhite} for the definition of tameness for varifolds}, and if $ \Sigma $ is connected and compact, then it is a geodesic sphere. In this direction we point out that unions of mutually tangent geodesic spheres cannot be tame varifolds in the sense of \cite{haslhoferOrWhite} (since the tangent cone at the singular point between the spheres is a multiplicity two plane). On the other hand, these configurations naturally arise as limits of sequences of connected and compact smooth hypersurfaces with mean curvatures converging to a constant  (see \cite{ciraolomaggi2017} and references therein). Hence, Theorem \ref{thm main hyperbolic}, in addressing rigidity and compactness in the hyperbolic space under assumptions that do not prevent bubbling, provides a useful improvement of \cite[Theorem 1.8]{haslhoferOrWhite}.
	\end{remark}

\subsection{Further directions} Heintze--Karcher type inequalities for smooth strictly mean convex hypersurfaces are proved in general substatic manifolds with horizon boundary in \cite{LiXiaJDG,fogagnoloPinamonti}. These results, whose proof relies on an elliptic method rather than on geodesic flows, generalize those proved in \cite{brendle}. More recently, the analysis of the equality case for these more general inequalities  in \cite{BorghiniFogagnoloPinamonti} leads to the extension of the rigidity statement of constant mean curvature hypersurfaces in \emph{all} substatic warped products, namely those satisfying (H0)-(H3), hence completely dropping the hypothesis (H4). In view of this recent development, it is natural to ask if our rigidity result, Theorem \ref{thm main}, can be also proved in this more general class of warped product manifolds. Attacking  this problem  by employing an approach similar to that of \cite{BorghiniFogagnoloPinamonti}, requires an elliptic method to  be developed in the singular setting of Theorem \ref{thm main}.

\medskip

Another natural question is that of obtaining quantitative estimates for almost-CMC boundaries, both in the Brendle class and on space forms. In this direction we mention the recent results on space forms \cite{ciraoloVezzoniIND,ciraoloRoncoroniVezzoni}, where, based on the moving planes method, sharp decays are obtained under a ``bubbling-preventing'' exterior/interior ball assumption. Non-sharp quantitative estimate have been derived in \cite[Theorem 1.3]{scheuer} on space forms, and in \cite[Theorem 1.4]{scheuerXia22} on a sub-class of Brendle's class (which still includes the model manifolds from General Relativity). Both these results require bounds on the $C^{2,\beta}$-geometry of the considered boundaries (in addition to interior ball conditions), and the resulting stability constants (together with the non-sharp stability exponents in the case of \cite[Theorem 1.4]{scheuerXia22}) depend to the particular $\beta$ under consideration. For this reason the compactness result in Theorem \ref{thm compactness 1} is entirely new even on smooth boundaries, as it does not require any uniform control on their geometry.

\medskip

\noindent {\bf Acknowledgements:} FM wishes to thank Claudio Arezzo for having introduced him to the framework considered in this work. FM is supported by NSF-DMS RTG 1840314, NSF-DMS FRG 1854344, NSF-DMS 2000034, and NSF-DMS 2247544.  MS acknowledges support of the INDAM-GNSAGA project "Analisi Geometrica: Equazioni alle Derivate Parziali e Teoria delle Sottovarietà" and PRIN project no.\ 20225J97H5.

\section{Sets of positive reach in Riemannian manifolds}\label{section positive reach} In this section $(M,g)$ is a {\it complete} Riemannian manifold of dimension $n$, with exponential function $\exp$ and Riemannian distance $d$, and we denote by $\Psi:TM\times[0,\infty)\to M$ the map
\begin{equation}
  \label{the map Psi}
  \Psi(p,\eta,t)=\exp(p,t\,\eta)\,,\qquad \forall (p,\eta)\in TM\,,t\ge0\,.
\end{equation}
We denote by $\der f$ the differential of $f:M\to\R$, and by $\nabla f$ and $\Der^2 f$ the gradient and Hessian of $f$ with respect to $g$. We use $\nabla$ also to denote the metric connection of $(M,g)$. A segment in $M$ is a unit speed geodesic $\g:[a,b]\to M$ such that $d(\g(a),\g(b))=b-a$.

\subsection{Normal bundles of closed sets}\label{section normal bundles} Given a closed set $\Gamma\subset M$, the projection map on $\Gamma$ is defined for at $p\in M$ as the subset of $\Gamma$ given by
\[
\xi_\Gamma(p)=\big\{a\in\Gamma:\dist(p,\Gamma)=d(p,a)\big\}\,.
\]
The {\bf unit normal bundle} $\mathcal{N}^1\Gamma$ and the {\bf normal bundle} $\mathcal{N}\Gamma$ of $\Gamma$ are defined by setting, for $a\in\Gamma$,
\begin{eqnarray*}
\mathcal{N}^1_a\Gamma&=&\big\{\eta\in T_aM:|\eta|=1\,,\mbox{$\exists s>0$ s.t. $s=\dist\big(\exp(a,s\,\eta),\Gamma\big)$}\big\}\,,
\\
\mathcal{N}_a\Gamma&=&\big\{t\,\eta:t\ge0\,,\eta\in\mathcal{N}^1_a\Gamma\big\}\,.
\end{eqnarray*}
We define $\rho_\Gamma:\mathcal{N}^1\Gamma\to(0,\infty]$, $A_\Gamma\subset TM$, and ${\rm Cut}(\Gamma)\subset M$ by setting
\begin{eqnarray}
\label{def rho sigma}
&&\rho_\Gamma(x,\eta)=\sup\big\{s>0:s=\dist(\exp(x,s\,\eta),\Gamma)\big\}\,,
\\
\label{def A Gamma}
&& A_\Gamma=\big\{(p,\eta,t):(p,\eta)\in\mathcal{N}^1\Gamma\,,t \in (0, \rho_\Gamma(p, \eta))\big\}\,,
\\
\label{def cut sigma}
&&{\rm Cut}(\Gamma)=\big\{\exp(x,s\,\eta):(x,\eta)\in\mathcal{N}^1\Gamma\,,s=\rho_\Gamma(x,\eta)<\infty\big\}\,,
\end{eqnarray}
so that, when $\Gamma$ is a closed $C^2$-hypersurface in $M$, $\rho_\Gamma$ is continuous on $\Gamma$, ${\rm Cut}(\Gamma)$ corresponds to the usual notion of {\bf cut-locus} of $\Gamma$ and satisfies $\H^n({\rm Cut}(\Gamma))=0$, and $\Psi|_{A_\Gamma}$ is a diffeomorphism between $A_\Gamma$ and
\[
\mathcal{U}(\Gamma)=\Psi(A_\Gamma)=\big\{\exp(a,s\,\eta):(a,\eta)\in\mathcal{N}^1\Gamma\,,s\in\big(0,\rho_\Gamma(a,\eta)\big)\big\}\,.
\]
The following theorem relies on a series of recent results \cite{albano,mennesantilli19,santilliAMPA20,kolasinski2021regularity,HugSantilli2022} to adapt/extend the above classical facts to the case when $\Gamma$ is merely a closed set.

\begin{theorem}\label{thm normal bundles closed sets}
	If $(M^n,g)$ is a complete Riemannian manifold and $\Gamma\subset M$ is closed, then:
	\begin{enumerate}
		\item[(i)] $\mathcal{N}\,\Gamma$ is a countably $ n $-rectifiable Borel subset of $  TM $;
		\item[(ii)]  $\mathcal{N}_a\Gamma$ is a convex cone in $ T_a M$ for every $ a \in \Gamma $;
		\item[(iii)] for each $ m = 0, \ldots, n-1 $ the set
		\begin{equation*}
			\Gamma^{(m)}: =  \{ a \in \Gamma : \dim \mathcal{N}_a \Gamma = n-m  \}
		\end{equation*}
		is countably $(\H^m,m)$-rectifiable;
        in particular\footnote{Indeed, if $\H^0(\mathcal{N}^1_a\Gamma) > 2$, then the convexity of  $\mathcal{N}_a\Gamma$ implies that $\dim \mathcal{N}_a \Gamma \ge 2$, i.e. $a\in\Gamma^{(m)}$ for some $m\le n-2$.}
		\begin{equation}
        \label{utile}
                \H^{n-1}(\{a \in \Gamma: \H^0(\mathcal{N}^1_a\Gamma) > 2\}) =0\,;
        \end{equation}
		\item[(iv)] $ \mathcal{N}^1 \Gamma $ is a countably $ (n-1) $-rectifiable Borel subset of $ T M $ and $ \rho_\Gamma : \mathcal{N}^1 (\Gamma) \rightarrow (0, +\infty]$ is an upper-semicontinuous function;
		\item[(v)] $\H^n({\rm Cut} (\Gamma)) =0 $.
		\item[(vi)] For each $ p \in \mathcal{U}(\Gamma) $ there exists a unique $ (a, \eta)\in TM $, $ |\eta|=1 $,  such that $ d(p,a) = \dist(p, \Gamma) $ and $ \exp(a, d(p,a)\eta) = p $;
		\item[(vii)] If $ \tau_0 > 0 $ and $ \rho_\Gamma(a, \eta) \geq \tau_0 $ for
        $ \H^{n-1} $ a.e.\ $(a, \eta)\in \mathcal{N}^1\Gamma $, then
		$$
        \big\{x \in M : 0 < \dist(x, \Gamma) < \tau_0\big\} \subseteq \mathcal{U}(\Gamma)\,.
        $$
	\end{enumerate}
\end{theorem}

\begin{proof} Let $\phi : U\subset M \rightarrow V\subset\R^n$ be a local chart of $M$. For each $a\in U$ there is a unique symmetric bijective linear map $S_a:\R^n\to\R^n$ such that
\[
\langle u,v\rangle_g=\der\phi(a)(v)\,\cdot\,S_a\big[\der\phi(a)(u)\big]\,,\qquad\forall u,v\in T_aM\,.
\]
A routine argument shows that the map $\Phi:TU\to V\times\R^n$ defined by $\Phi(a,v)=(\phi(a),S_a[\der\phi(a)(v)])$, $(a,v)\in TU$, is a diffeomorphism, with
\[
\Phi\big(\mathcal{N}\,\Gamma\,\llcorner U\big)=\big(\mathcal{N}\,\ov{\phi(\Gamma\cap U)}\big)\,\llcorner V\,.
\]
Similarly, for every $a\in \Gamma$ and $m=0,...,n-1$, we have
\[
[S_a\circ\der\phi(a)]\big(\NN_a\Gamma\big)=\NN_{\phi(a)}\,\ov{\phi(\Gamma\cap U)}\,,\quad \phi(\Gamma^{(m)}\cap U)=[\ov{\phi(\Gamma\cap U)}]^{(m)}\cap V\,.
\]
Hence, conclusions (i), (ii), and (iii) follow from the analogous statements in the Euclidean case proved in \cite[Remark 4.3]{santilliAMPA20} and \cite{mennesantilli19}. The upper-semicontinuity of $ \rho_\Gamma $ claimed in (iv) follows by an obvious adaptation of the argument for the Euclidean space proved in \cite[Lemma 2.35]{kolasinski2021regularity},  while we can use \cite[3.2.31]{FedererBOOK} in combination with (i) to infer that $ \mathcal{N}^1(\Gamma) $ is countably $(n-1)$-rectifiable. Conclusion (v) is proved in \cite[Theorem 1]{albano}. Conclusion (vi) follows from the remark that if $ s > 0 $ and $ \alpha $ is a unit-speed geodesic such that $ \alpha(0) \in \Gamma $ and $ \dist(\alpha(s), \Gamma) = s $, then, for every $ 0 \leq t < s $, we have $\dist(\alpha(t), \Gamma) = t$, $\xi_\Gamma(\alpha(t)) = \{\alpha(0)\}$, and $ \alpha|_{[0,t]}$ is the unique segment joining $ \alpha(0) $ and $ \alpha(t) $. We are left to prove conclusion (vii), which requires modifications to the proof of \cite[Lemma 3.19]{HugSantilli2022}.  With $\Psi$  as in \eqref{the map Psi}, we define $$ Q^\ast = \{ (a,\eta,t) : (a, \eta) \in \mathcal{N}^1\Gamma, \; 0 < t < \inf\{\tau_0, \rho_\Gamma(a,\eta) \} \} $$ and we notice that $ \Psi (Q^\ast) =  \{x \in M : 0 < \dist(x,\Gamma)< \tau_0\} \setminus {\rm Cut}(\Gamma) $. Let $$ Q = Q^\ast \cap \{(a,\eta,t) : \rho_\Gamma(a,\eta) \geq \tau_0 \} $$ so that $ \H^n(Q^\ast \setminus Q) =0 $ by assumption. Hence $ \H^n(\Psi(Q^\ast) \setminus \Psi(Q)) =0 $, and, by (v), we conclude that $ \Psi(Q) $ is dense in  $ \{x \in M : 0 < \dist(x,\Gamma) < \tau_0\} $. Let $ x \in M $ and $ t = \dist(x, \Gamma) $ with $ 0 < t < \tau_0 $. There exists a sequence $(a_i, \eta_i, t_i) \in Q $ such that $ \Psi(a_i, \eta_i, t_i) \to x $. Since
$ d(a_i,x) \leq d(a_i, \Psi(a_i, \eta_i, t_i)) + d(\Psi(a_i, \eta_i, t_i),x) $ for $ i \geq 1 $ and $ t_i = \dist(\Psi(a_i, \eta_i, t_i), \Gamma) \to t $, we infer that $ \limsup_{i \to \infty} d(a_i,x) \leq t $. By compactness, there are  $(a, \eta) \in TM $ with $|\eta|=1 $ and a subsequence $(a_{i_j}, \eta_{i_j})$ such that $(a_{i_j}, \eta_{i_j}) \to (a, \eta) $ as $ j \to \infty $. It follows that
$$ t = \lim_{j \to \infty} \dist(\Psi(a_{i_j}, \eta_{i_j}, t_{i_j}), \Gamma) = \dist(\Psi(a,\eta,t), \Gamma) $$
and $(a, \eta)\in \mathcal{N}^1\Gamma $. By conclusion (iv) and by $\rho_{\Gamma}(a_i, \eta_i) \geq \tau_0$ we find $ \rho_\Gamma(a, \eta) \geq\tau_0$, and thus $x \in \mathcal{U}(\Gamma) $.
\end{proof}

\subsection{Sets of positive reach}\label{subsection positive reach} Introduced in the Euclidean setting by Federer \cite{MR0110078}, sets of positive reach have been studied in the Riemannian setting by Kleinjohann \cite{Kleinjohann} and Bangert \cite{Bangert}. Given a closet set $\Gamma\subset M$, the {\bf set of unique projection over $\Gamma$} is defined as
\[
\Unp(\Gamma) =\big\{x \in M : \H^0(\xi_\Gamma(x))=1\big\}\,.
\]
Given $a\in\Gamma$, we denote by $\reach(\Gamma, a)$ the supremum of those $r\ge0$ such that $B(a,r) \subseteq \Unp(\Gamma)$, and say that $ \Gamma $ is a {\bf set of locally positive reach} if $\reach(\Gamma,a)>0$ for each $a\in\Gamma$, and is a {\bf set of positive reach} if $\reach(\Gamma,\cdot)$ is uniformly positive on $\Gamma$.

\begin{remark}\label{remark pos reach}
  {\rm If $\Gamma$ is closed and $\rho_\Gamma(x,\eta)\geq \tau_0>0$ for $\H^{n-1}$ a.e. $(x,\eta)\in\mathcal{N}^1\Gamma$, then by Theorem \ref{thm normal bundles closed sets}-(vii),
  \[
  \{x\in M:0<\dist(x,\Gamma)<\tau_0\}\subset\mathcal{U}(\Gamma) \subseteq \Unp(\Gamma)\,.
  \]
  In particular, $\Gamma$ is of positive reach, with $ \reach(\Gamma,\cdot) \geq \tau_0$ on $\Gamma$.}
\end{remark}

The following result, contained in \cite{Kleinjohann}, is crucial in obtaining \eqref{pino} and \eqref{prop: hk and n,h sets: eq12} in the proof of Theorem \ref{thm hk inq}.

\begin{theorem}[Kleinjohann]\label{thm kleinjohann}
	If $ A\subseteq M $ is a set of positive reach with compact boundary, then there exists $ \e(A)> 0 $  such that for every $t\in(0,\e(A))$ the set
	\[
	A_t = \{ x  \in M : \dist(A,x) = t \}
	\]
	is a compact $C^{1,1} $-hypersurface contained in $ \Unp(A) $, and the geodesic-flow map $ \Phi_t : \mathcal{N}^1(A) \rightarrow A_t $, defined by $ \Phi_t(a, \eta) = \exp(a, t\eta) $ for $ (a, \eta)\in \mathcal{N}^1(A) $, is bi-Lipschitz on $\mathcal{N}^1(A)$. In particular, $ \mathcal{N}^1(A) $ is an $ (n-1) $-dimensional compact Lipschitz submanifold of $T M$.
\end{theorem}

\begin{proof} Let $ \mathcal{U} $ be an open neighborhood of the null section of $ TM $ and $ \mathcal{V} $ an open subset of $ M \times M $ such that the map $ \mathcal{U} \ni (a,v) \mapsto (a, \exp(a,v)) $ is a diffeomorphism of $ \mathcal{U} $ onto $ \mathcal{V} $. Let $ \Phi $ be its inverse, so that $ \exp(p, \Phi(p,q))  = q $ for every $(p,q) \in \mathcal{V} $. As explained in \cite[middle of page 336]{Kleinjohann}, one can choose for each $ a \in A $ a number $ \zeta(a) > 0 $\footnote{Denoted with $ z'(a,\epsilon) $ in \cite[page 336]{Kleinjohann}.} so that
	$$ B(a, \zeta(a)) \times B(a, \zeta(a))  \subseteq \mathcal{V} \quad \textrm{for $ a \in A $} $$
and, defining	 $ W : = \bigcup_{a \in A} B(a, \zeta(a))  $, we have that for every $ x \in W $ there exists a unique minimizing geodesic joining $ x $ and $ A $ and $ \xi_A $ is locally lipschitz on $ W $ (see \cite[Satz (2.5)]{Kleinjohann}). For $ x \in W \setminus A $ we set
	 $$ \nu(x) = \frac{\Phi(x, \xi_A(x))}{| \Phi(x, \xi_A(x)) |} \quad \textrm{and} \quad  \eta(x) = \frac{\Phi(\xi_A(x),x)}{| \Phi(\xi_A(x),x) |}. $$
	   It follows from \cite[Satz (2.1) and Satz (2.3)]{Kleinjohann} that $ \dist(A, \cdot) $ is continuously differentiable on $ W \setminus A $, with $ \nabla \dist(A,x) = - \nu(x) $ for $ x \in W \setminus A $. In particular $ \nabla \dist(\cdot, A) $ is locally Lipschitz on $ W \setminus A $. Since $ \dist(A,x) = | \Phi(\xi_A(x),x)| $ for $ x \in W $, it follows that
	   $$ \Psi_t(\xi_A(x), \eta(x)) = \exp(\xi_A(x), \Phi(\xi_A(x),x)) = x $$
	   for $ x \in W \cap A_t $. By applying the Lebesgue covering lemma to  $ \{B(a, \zeta(a)) : a \in \partial A\} $, we find a positive number $ \e(A) > 0 $ such that $ \{x \in M : \dist(A,x)< \e(A)\} \subseteq W$, and thus conclude the proof.
\end{proof}

Following a standard convention (see, e.g. \cite{Grove}) we say that $p\in M\setminus\Gamma$ is a {\bf critical point for $ \dist(\cdot, \Gamma)$} if for every $ v \in T_p(M) $, $ v \neq 0 $, there exists $ a \in \Gamma $ with $ d(p,a)= \dist(p, \Gamma) $ and a segment $ \gamma : [0,d(p,a)] \rightarrow M $  such that $ \gamma(0) = p $, $ \gamma(d(p,a)) = a $ and ${\rm angle}(v,\gamma'(0))\le\pi/2$. Correspondingly, $ \tau > 0 $ is a {\bf regular value for $ \dist(\Gamma, \cdot)$} if there are no critical points $p$ of $\dist(\cdot,\Gamma)$ with $\dist(p,\Gamma)=\tau$. The following result is a special case of the main result obtained by Bangert in \cite{Bangert}, and plays an important role in our analysis (see the proof of \eqref{has positive reach}).

\begin{theorem}[Bangert]
\label{thm bangert}
If $ \Gamma \subseteq M $ is compact and $ \tau > 0 $ is a regular value of $ \dist(\cdot,\Gamma) $, then $\{x \in M: \dist(x,\Gamma) \geq \tau \} $ is a set of positive reach.
\end{theorem}

\begin{proof}
By \cite[Proposition 3.4]{MR1941909}, $ f = -\dist(\Gamma, \cdot) $ is locally semiconvex on $ M \setminus \Gamma$. In particular, $ f $ belongs to the class $ \mathcal{F}(M \setminus \Gamma) $ introduced by Bangert in \cite{Bangert}. Moreover, by \cite[Lemma 5.5]{MR2954647}, $ p \in M \setminus \Gamma $ is a regular point of $ f $ if and only if there exists $ v \in T_p M $ such that\footnote{The function $f$ being semiconcave, it may fail to have a differential at $p$. However, the limit $\partial_pf(v)$ will exist for every $p$ and $v$. Here are we using the same notation found in \cite{Bangert}.} $\partial_pf(v)=\lim_{t\to 0^+}(f(p+t\,v)-f(p))/t$ is negative. It follows that all points of $ f^{-1}(-\tau) $ are regular in the sense of \cite[Definition (i)]{Bangert}, so that $ f^{-1}((-\infty, -\tau]) $ is a set of locally positive reach by the main theorem of \cite{Bangert}. Since $ (M,g) $ is complete, $ f^{-1}(-\tau) = \partial[ f^{-1}((-\infty, -\tau])] $ is compact, and thus we conclude by the general fact that, if $ A\subseteq M $ is a set of locally positive reach and $ \partial A $ is compact, then $ A $ is a set of positive reach (for example because, by \cite[Lemma 1.1]{Kleinjohann}, $ \reach(A, \cdot) $ is continuous on $ A $).
\end{proof}

\section{A Lusin-type property of White's $(m,\l)$-sets}\label{section viscous mc} In Section \ref{section white} we recall a viscosity formulation of the notion of ``being $m$-dimensional with mean curvature vector bounded by $\lambda$'' for closed subsets of a Riemannian manifold $(M,g)$, as introduced by White in \cite{whiteMC}. Then, in Section \ref{section mario}, we extend from the Euclidean to the Riemannian setting a ``Lusin condition'' for normal bundles proved in \cite{SantilliBMS}.

\subsection{White's $(m,\lambda)$-sets}\label{section white} Given an integer $m\in\{1,...,n-1\}$ and a constant $\lambda\ge 0$, we say that a closed subset $\Gamma$ of a Riemannian manifold $(M,g)$ is a {\bf White $(m,\lambda)$-set} in $(M,g)$ if, for every $ f\in C^2(M)$ such that $f|_{\Gamma}$ admits a local maximum at $x\in\Gamma$, it holds that
	\begin{equation}\label{white definition}
		{\rm trace}_m(\Der^2 f (x)) \leq \lambda |\nabla f (x)|\,.
	\end{equation}
Here ${\rm trace}_m(\Der^2 f (x))=\l_1+\cdots+\l_m$ if $\l_1\le\l_2\le\cdots\le\l_n$ are the eigenvalues of $\Der^2 f (x)$ listed in increasing order. A fundamental result concerning White $(m,\l)$-sets relates condition \eqref{white definition} to the notion of {\it distributional} mean curvature for a varifold. This theorem plays a key role in our analysis (specifically, it allows us to use Theorem \ref{thm hk inq} in proving Theorem \ref{thm rigidity}, see the next section).

\begin{theorem}[White]\label{thm white}
  If $\l\ge0$, $V$ is an $m$-dimensional varifold in $(M,g)$, and $\vec{H}$ is a Borel vector field in $M$ such that $\|\vec{H}_V\|_{L^\infty(\spt\|V\|)}\le\l$ and
  \[
  \int (\Div^\tau X)(x)\,dV(x,\tau)=-\int_M\,\langle\vec{H},X\rangle_g\,d\|V\|\,,
  \]
  for every $X\in\X(M)$,  then $\spt\,\|V\|$ is a White $(m,\l)$-set in $(M,g)$.
\end{theorem}

\begin{proof}
  This is \cite[Corollary 2.8]{whiteMC}.
\end{proof}

%
%
%
We shall also need the following simple fact:

\begin{lemma}\label{lem: mh sets and conformal metric}
If $(M,g)$ is a Riemannian manifold, $\Gamma$ is a White $(m,\lambda)$-set in $(M,g)$, $\vphi\in C^\infty(M)$, $ g^{*} = e^{2\,\vphi} g $, and
\[
C(\Gamma,\vphi)=\sup_{\Gamma}| \nabla \vphi|<\infty\,,
\]
 then $\Gamma$ is a White $(m, \lambda + 3\,m\,C(\Gamma,\vphi))$-set in $(M,g^{*})$.
\end{lemma}

\begin{proof} Denoting with $ \nabla^* $ and $ \Der^2_* $ the Riemannian connection and the Hessian operator with respect to $ g^{*} $, it is enough to prove that
\begin{equation}\label{lem: mh sets and conformal metric: eq1}
	\Der^2_* f(x) \leq \Der^2 f(x) + 3\, |\nabla \vphi(x)|\, |\nabla f(x)|\, g_x\,
\end{equation}
whenever $x\in O$, $O$ is open in $M$, and $ f\in C^2(O)$. To this end, by \cite[II, Proposition 3.9]{MR1390760}, we compute
\begin{eqnarray*}
  	&&\Der^2_*f(U, V) = U(V(f)) - (\nabla^*_U V)(f)
  = \Der^2 f(U, V) - g(\nabla \vphi, U) g(\nabla f, V)
  \\
  &&\hspace{5cm}- g(\nabla \vphi, V) g(\nabla f, U) - g(U,V)g(\nabla \vphi, \nabla f)\,,
\end{eqnarray*}
whenever $U,V\in\X(M)$. Hence,
\[
\Der^2_* f(U,U) \leq \Der^2f (U,U)  + 3 | \nabla f | | \nabla \vphi| | U |^2\,.
\]
If now $O$ is a neighborhood of some $x\in\Gamma$ and $f|_\Gamma$ has a local maximum at $ x $, then this last inequality, combined with \cite[Lemma 12.3]{whiteMC} and the fact that $\Gamma$ is a White $(m,\lambda)$-set in $(M,g)$, implies that
	\begin{flalign*}
		{\rm trace}_m(\Der^2_*f(x)) & \leq {\rm trace}_m\big(\Der^2f(x) + 3\,|\nabla f(x)|\,|\nabla \vphi(x)| \,g_x\big)\\
		& = {\rm trace}_m(\Der^2f(x)) + 3\,m\, |\nabla f(x)|\,|\nabla \vphi(x)|\\
		& \leq (\lambda + 3\,m\,C(\Gamma,\vphi))\, |\nabla f(x)|\,,
	\end{flalign*}
thus completing the proof.
\end{proof}

\subsection{A Lusin-type property of  normal bundles}\label{section mario} The following theorem is needed at a crucial step \eqref{prop: hk and n,h sets: eq1} in the proof of Theorem \ref{thm hk inq}.

\begin{theorem}[Lusin-type property of normal bundles]\label{thm Santilli riem}
	Suppose $ \Gamma \subseteq M $ is a White $ (m,\lambda)$-set in $ (M,g) $ such that $ \Gamma $ is a countable union of sets with finite $ \H^m $-measure. Then the following implication holds:
\begin{equation}\label{manifold case Z}
	Z \subseteq \Gamma, \; \H^m(Z \cap \Gamma^{(m)}) = 0 \quad \implies \quad
\left\{
\begin{split}
&\H^{n}(\mathcal{N}(\Gamma) \llcorner  Z) =0\,,
\\
&\H^{n-1}(\mathcal{N}^1(\Gamma) \llcorner  Z) =0\,.
\end{split}
\right .
\end{equation}
\end{theorem}

\begin{proof} By \cite[3.3, 3.8, 3.9]{SantilliBMS}, if $ W \subseteq \R^k $ is open, $ \Gamma $ is a White $ (m,\lambda) $-set in $ W  $, and $\Gamma$ is a countable union of sets with finite $ \H^m $-measure, then
\begin{equation}\label{Euclidean case Z}
Z \subseteq \Gamma, \; \H^m(Z \cap \Gamma^{(m)}) = 0 \quad \implies \quad
\left\{
\begin{split}
&\H^{k}\big(\mathcal{N}(\Gamma) \llcorner Z\big)=0 \,,
\\
&\H^{k-1}\big(\mathcal{N}^{\cred 1}(\Gamma) \llcorner Z\big)=0 \,.
\end{split}
\right .
\end{equation}
We now reduce the proof of \eqref{manifold case Z} to an application of the Euclidean case \eqref{Euclidean case Z}. To this end, by the Nash embedding theorem \cite[Theorem 3]{Nash1956}, we can directly assume that $ M $ is an $n$-dimensional embedded submanifold of the Euclidean space $\R^k$ (for some large $ k $) and $ g $ is the Riemannian metric on $ M $ induced by the Euclidean metric $ \langle \cdot, \cdot \rangle $ of $ \R^k $. 
If $ X \subseteq M $ is relatively closed  in $M$, then we denote by  $ \widetilde{\mathcal{N}}\overline{X}$ the normal bundle of the closure in $ \mathbb{R}^k $ of $ X$, while we keep the symbol $ \mathcal{N}X $ for the normal bundle of $ X $ as a subset of $(M,g) $ (notice $ \mathcal{N}X \subseteq T M $). If $ a \in M $, then $ \pi_a $ and $ \pi^\perp_a $ denote the orthogonal projections of $\R^k$ onto $ T_aM $ and $ T^\perp_aM $ respectively. We now divide the rest of the proof into two claims and one final argument.

\medskip

\noindent {\it Claim one}\,: If $ X \subseteq M $ is relatively closed  in $ M $, then
\begin{equation*}
\widetilde{\mathcal{N}}_a\overline{X}= \mathcal{N}_aX + T^\perp_aM\,,\qquad\forall a\in X\,.
\end{equation*}
Indeed,  let $ u \in \widetilde{\mathcal{N}}_a\overline{X}$ be such that $ \pi_a(u) \neq 0 $ and $ B \cap \overline{X} = \varnothing $, where, for some $r>0$, $ B $ is the open Euclidean ball in $ \R^k $ of radius $ r|u| $ centered at $  a+ ru $. Since $ u \in T^\perp_a(\partial B) $ and $ \langle \pi_a(u), u \rangle > 0 $, we conclude that $ \pi_a(u) \in T_aM \setminus T_a(\pa B) $. Then we choose an open neighborhood $V $ of $ a $ and a continuous function $ \eta : V \cap M \rightarrow \mathbf{S}^{k-1} $ such that $ \eta(a) = \pi_a(u)/|\pi_a(u)| $ and $ \eta(b) \in T_bM$ for every $ b \in V \cap M $. Since $ [b\in  V \cap M \cap \partial B] \mapsto  \dist(\eta(b), T_b(\pa B)) $ is continuous and $ \dist(\eta(a), T_a(\pa B)) > 0 $, there is an open neighborhood $ W \subseteq V $ of $ a $ such that $ \dist(\eta(b), T_b(\partial B)) > 0 $ for every $ b \in W \cap \partial B \cap M $. Hence,
$\dim \big[  T_bM \cap T_b(\pa B)\big] \leq n-1$ and $T_bM + T_b(\pa B) = \R^k$ for every $ b \in W \cap M \cap \partial B $. This means that the submanifolds $ W \cap \partial B $ and $ W \cap M $ are transversal and consequently
\begin{equation}\label{Riemannian Lusin condition 3}
\partial B \cap M \cap W \quad \textrm{is an $n-1 $-dimensional smooth submanifold of $ M $}
\end{equation}
with $T_bM \cap T_b(\pa B)= T_b(M \cap \partial B) $ for $ b \in W \cap M \cap \partial B $; see \cite[pp.\ 29-30]{MR0348781}. Now, observing that
if $ v \in T_a(\partial B \cap M) = T_a(\partial B) \cap T_a(M) $, then
$ \langle \pi_a^\perp(u), v\rangle =0  $ and $\langle \pi_a(u), v \rangle =  \langle u, v \rangle =0 $, we find that
\begin{equation}\label{Riemannian Lusin condition 4}
 \pi_a(u) \in T^\perp_a(\partial B \cap M)\,.
\end{equation}
 For $t$ sufficiently small, since $ \langle \pi_a(u), u \rangle > 0 $, we notice that $ \exp(a, t\pi_a(u)) \in B \cap M  $  and we conclude from \eqref{Riemannian Lusin condition 4} that the open geodesic ball of $(M,g)$ centred at $ \exp(a, t\pi_a(u)) $ and radius $ t|\pi_a(u)| $ is contained in $ B \cap M $. This means, always for $ t $ sufficiently small,  that $ d(\exp(a, t\pi_a(u)), X) = t|\pi_a(u)| $ (thanks to $ B \cap X = \varnothing $) and  $ \pi_a(u) \in \mathcal{N}_a(X) $.

\medskip

We have thus proved  $\widetilde{\mathcal{N}}_a\overline{X} \subseteq \mathcal{N}_aX + T^\perp_aM $ for $ a \in X $. To prove the opposite inclusion, let now $ v \in \mathcal{N}_aX $ with  $ v \neq 0 $. An open geodesic ball $ G $ in $ M $ with sufficiently small radius is such that $ \partial G $ is a smooth $ n-1 $ dimensional submanifold in $ \R^k $, $ G\cap X= \varnothing $, $ a \in \partial G $, and $ v $ is an interior normal of $ G $ at $ a $. Then there exists $ r > 0 $ such that the open Euclidean ball $ B $ in $ \R^k $ of radius $ r $ centered at $ a + r v $ satisfies $ B \cap \partial G = \varnothing $ and $ B \cap G \neq \varnothing $. Choosing $ r $ smaller if necessary, we can also ensure that $ B $ is contained in the tubular neighbourhood of $ M $ where the nearest point projection onto $ M $ is single valued. This means that $ B \cap M $ is a connected subset of $ M \setminus \partial G $; since $ G $ is a connected component of $ M \setminus \partial G $ and $ G \cap B \neq \varnothing $, we infer that $ B \cap M \subseteq G $. The latter inclusion implies that $ B \cap X = \varnothing $. It follows that  $ B \cap \overline{X} = \varnothing $ and $ \mathcal{N}_aX \subseteq \widetilde{\mathcal{N}}_a\overline{X}$. Now the convexity of $ \widetilde{\mathcal{N}}_a\overline{X} $ and the obvious inclusion $ T_a^\perp M \subseteq \widetilde{\mathcal{N}}_a\overline{X} $  allow to conclude
$\mathcal{N}_aX + T^\perp_a M \subseteq \widetilde{\mathcal{N}}_a\overline{X} + \widetilde{\mathcal{N}}_a\overline{X} = \widetilde{\mathcal{N}}_a\overline{X}$ and to complete the proof of claim one.

\medskip

\noindent {\it Claim two}: If $W\subseteq \R^k$ is an open set such that $ \overline{W} \cap M $ is compact, then there is $\lambda_W\ge0$ such that $ \Gamma \cap W$ is a White $(m, \lambda_W)$-set in $W$. Firstly, notice that $ M \cap W $ and $ \Gamma \cap W $ are relatively closed in $ W $. Now, let $b \in \Gamma \cap W$ and $f \in C^2(W)$ such that $ f|_\Gamma$ has a local maximum at $ b$. Denoting by $\overline{\nabla}$ both the Euclidean gradient operator and metric connection, and by $\overline{\Der}^{\,2}$ the Euclidean Hessian operator, we set
$\eta(x) = \pi^\perp_x(\overline{\nabla} f(x))$ for $ x \in W \cap M$, notice that $\pi_x(\overline{\nabla} f(x))=\nabla f(x)$, and compute
\begin{flalign*}
\overline{\Der}^{\,2} f(x)(v,w)	& = \langle \overline{\nabla}_v \overline{\nabla} f (x), w \rangle  = \langle \overline{\nabla}_v \nabla  f(x), w \rangle + \langle \overline{\nabla}_v \eta(x) , w \rangle \\
	& = \Der^2 f(x)(v, w) + \langle A_{\eta(x)}(v), w \rangle   = \Der^2 f(x)(v, w) - Q(x)(v,w)
\end{flalign*}
for every $ x \in M \cap W $ and $ v , w \in T_x(M) $, where $ A_{\eta(x)}: T_x M \rightarrow T_x  $ is the shape operator of $ M $ in the direction $ \eta(x) $, $ Q(x) : T_x M \times T_x M \rightarrow \R$ is the symmetric bilinear form defined as $Q(x)(v,w)= \langle S(v,w), \eta(x)  \rangle $ and $ S $ is the second fundamental form of $ M $. If $ Q(b) (v,v)\geq 0 $ for all $ v \in T_bM$ then we obtain from \cite[Lemma 2.3]{MR1996769} and \cite[Lemma 12.3]{whiteMC}
\begin{flalign*}
	{\rm trace}_m(\overline{\Der}^{\,2} f(b)) & \leq {\rm trace}_m\big[ \overline{\Der}^{\,2}f(b)| T_bM\times T_bM\big] \\
	& = {\rm trace}_m [\Der^2f(b) - Q(b)]\\
	&  \leq	{\rm trace}_m(\Der^2f(b)) \leq \lambda | \nabla f(b)| \leq \lambda | \overline{\nabla} f(b)|.
\end{flalign*}
Now we assume that $ Q(b)(v,v) < 0 $ for some $ v \in T_bM$.  Then we define
\begin{equation*}
	\mu_W =  \inf\{ \langle S(u,u), \nu \rangle: x \in \overline{W} \cap M,\, u \in T_xM, \, \nu \in (T_xM)^\perp, \, |u|=|\nu|=1 \}
\end{equation*}
and we notice that $ -\infty < \mu_W   < 0 $ and
\begin{equation*}
	Q(b)(v,v) \geq \mu_W  | \eta(b)| \langle v, v \rangle \qquad \textrm{for $ v \in T_bM$.}
\end{equation*}
Therefore, by  \cite[Lemma 2.3]{MR1996769} and \cite[Lemma 12.3]{whiteMC}
\begin{flalign*}
{\rm trace}_m(\Der^2f(b)) & = {\rm trace}_m\big[ \overline{\Der}^{\,2} f(b)|( T_b M \times T_b M) + Q(b) \big] \\
& \geq {\rm trace}_m\big[ \big(\overline{\Der}^{\,2}f(b) + \mu_W| \eta(b)| \langle \cdot, \cdot \rangle \big)\big| T_bM\times T_b(M)\big] \\
& \geq {\rm trace}_m\big[\overline{\Der}^{\,2}f(b) + \mu_W| \eta(b)| \langle \cdot, \cdot \rangle\big] \\
& = {\rm trace}_m(\overline{\Der}^{\,2}f(b)) + m\mu_W | \eta(b) |,
\end{flalign*}
and we deduce that
\begin{equation*}
 {\rm trace}_m(\overline{\Der}^{\,2}f(b)) \leq \lambda | \nabla f(b)|- m \mu_W | \eta(b)| \leq (\lambda - m \mu_W)| \overline{\nabla} f(b) |.
\end{equation*}

\medskip

\noindent {\it Conclusion of the proof}:  Let $ Z \subseteq \Gamma $ such that $\H^m(Z \cap\Gamma^{(m)})=0$. Since $ \mathcal{N}_a(\Gamma) \subseteq T_a M $ it follows from claim one that
\begin{equation*}
\dim \widetilde{\mathcal{N}}_a\overline{\Gamma} = \dim \mathcal{N}_a\Gamma + \dim T^\perp_aM \quad \textrm{for $ a \in \Gamma $,}
\end{equation*}
so that
\begin{equation*}
	\H^m(Z \cap \{ a :  \dim \widetilde{\mathcal{N}}_a\overline{\Gamma} = k-m  \}) =0.
\end{equation*}
It follows from claim two and \eqref{Euclidean case Z} that $ \H^{k}(\widetilde{\mathcal{N}}\,\overline{\Gamma}\llcorner Z) =0 $. Let $ P : M \times \R^k \rightarrow T M $ be the smooth map defined by
\begin{equation*}
P(a,u) = (a, \pi_a(u)) \quad \textrm{for $(a,u) \in M \times \R^k $.}
\end{equation*}
Since  $ P(\widetilde{\mathcal{N}}\,\overline{\Gamma}\llcorner \Gamma) = \mathcal{N}\,\Gamma $ by claim one, noting by Theorem \ref{thm normal bundles closed sets} that $\widetilde{\mathcal{N}}\,\overline{\Gamma} \llcorner \Gamma $ is a countably $ k $-rectifiable subset of $ M \times \R^k $ and $ \mathcal{N}\,\Gamma $ is a countably $ n $-rectifiable subset of $T M $. By the coarea formula \cite[3.2.22]{FedererBOOK},
\begin{equation*}
 0 = \int_{ \widetilde{\mathcal{N}}\,\overline{\Gamma}\llcorner Z} \textrm{ap} J_{n}P\, d\H^k = \int_{\mathcal{N}\Gamma \llcorner Z} \H^{k-n}( T^\perp_a M)\, d\H^n_{(a,u)}\,.
\end{equation*}
Since $ \H^{k-n}( T^\perp_a M) = + \infty $ for every $ a \in M $, we have $ \H^n(\mathcal{N}\Gamma\llcorner Z) =0 $.
\end{proof}

\section{Viscous Heintze--Karcher inequalities}\label{section HK} In Theorem \ref{thm hk inq} below, we prove the Heintze--Karcher inequalities of Brendle \cite{brendle}, and address their equality cases, in the viscous setting of White \cite{whiteMC}. Starting from this result, in Theorem \ref{thm rigidity}, we extend Brendle's rigidity theorem to the distributional setting. Throughout the section, $n\ge 3$ and $(M,g)$ is a Riemannian manifold which satisfies at least ${\rm (H0)}$--${\rm (H3)}$. 

We consider $\Sigma\subset M$ such that:

\medskip

\begin{enumerate}
\item[(A1)] $\Sigma$ is a smooth embedded hypersurface in $M$ with $\overline{\Sigma}\subset M^\circ$ and
  \[
  \H^{n-1}(\overline{\Sigma} \setminus \Sigma)=0\,,\qquad \H^{n-1}(\Sigma)<\infty\,.
  \]
\end{enumerate}


\noindent Notice carefully that {\it we do not assume $\Sigma$ to be closed}. Thus, $\ov\Sigma\setminus\Sigma$ may be non-empty and may contain singular points (i.e., $\ov\Sigma$ may fail to be an hypersurface at points in $\ov\Sigma\setminus\Sigma$), and $\Sigma$ may consists of countably many connected components. Our second main assumption is that $\ov{\Sigma}$ is (topologically) a boundary, namely,

\medskip

\begin{enumerate}
\item[(A2)] there is $\Om\subset M$ open such that
\begin{equation}
  \label{Omega open}
\mbox{either $\pa\overline{\Om}=\overline{\Sigma}$ or $\pa\overline{\Om}=\ov\Sigma\cup N_0$,}
\end{equation}
 and $ H_\Sigma = \vec{H}_\Sigma \cdot \nu_\Omega $ is positive on $ \Sigma $.
\end{enumerate}

\medskip

\noindent Under \eqref{Omega open}, assumption ${\rm (A1)}$ implies that $\Om$ is a set of finite perimeter in $M$ thanks to Federer's criterion, see \cite[4.5.12]{FedererBOOK}. In particular, if we denote by $\pa^*\Om$ the reduced boundary of $\Om$, and by $\nu_\Om$ its measure theoretic outer $g$-unit  normal, then we observe that $\Sigma\subset\pa^*\Om$ and $\nu_\Om$ is smooth on $\Sigma$. It thus makes sense to define $$H_\Sigma=\vec{H}_\Sigma\cdot\nu_\Om \quad \textrm{on $\Sigma$}, $$ where $\vec{H}_\Sigma$ is the mean curvature vector of $\Sigma$ in $(M,g)$.

\begin{theorem}\label{thm hk inq} If $n\ge 3$, $(M,g)$ satisfies ${\rm (H0)}$-${\rm (H3)}$, and the pair $(\Sigma,\Om)$ satisfies assumptions ${\rm (A1)}$, ${\rm (A2)}$, and

\medskip

\begin{enumerate}
\item[(A3)] for some $ \l \geq 0 $, $\overline{\Sigma} $ is a White $(n-1, \l)$-set in $ M $,
\end{enumerate}

\medskip

\noindent then, denoting by $r$ the projection of $M=N\times(0,\bar{r})$ over $(0,\bar{r})$, and setting
\[
f=h'\circ r\,,\qquad g^{*} =f^{-2}\, g\,,
\]
the following statements hold:
\begin{enumerate}
	\item[(a)] if $\pa\Om=\ov\Sigma$, then
\begin{equation}
  \label{conclusion a}
  (n-1)\int_{\Sigma}\frac{f}{H_{\Sigma}}\, d\H^{n-1} \geq n \int_{\Omega} f\, d\H^n\,;
\end{equation}
\item[(b)] if $\pa\Omega=\ov\Sigma\cup N_0$, then
\begin{equation}
  \label{conclusion b}
(n-1)\int_{\Sigma}\frac{f}{H_{\Sigma}}\, d\H^{n-1} \geq n \int_{\Omega} f\, d\H^n  + h(0)^n \vol(N, g_N)\,;
\end{equation}
\item[(c)] if either $(a)$ or $(b)$ holds with equality, then $ \Sigma $ is umbilic in $(M,g)$, $M\setminus\Omega$ has positive reach in $(M, g^{*})$, and\footnote{Here, given $\nu=(\tau,a)\in T_{(x,t)}M\equiv T_xN\times\R$, we have set $(g_N)|_{(x,t)}(\nu,\nu)=(g_N)_x(\tau,\tau)$.}
    \begin{equation}
      \label{for H3 star}
      \Big(2\,\frac{h''}h+(n-2)\,\frac{((h')^2-\rho)}{h^2}\Big)'\,g_N(\nu_\Om,\nu_\Om)=0\,,\qquad\mbox{on $\Sigma$}\,.
    \end{equation}
\end{enumerate}
\end{theorem}

\begin{theorem}
  \label{thm rigidity}  If $n\ge 3$, $(M,g)\in\B_n\cup\B_n^*$, and the pair $(\Sigma,\Om)$ satisfies assumptions ${\rm (A1)}$,
\medskip

  \begin{enumerate}
  \item[(A2)'] there is $\Om\subset M$ open such that either $\pa\overline{\Om}=\overline{\Sigma}$ or $\pa\overline{\Om}=\ov\Sigma\cup N_0$;
\medskip
  \item[(A3)'] there is $H_0\ge0$ such that, for every $Y\in\X(M)$,
  \begin{equation}
  \label{vale div thm}\int_\Sigma\Div^\Sigma Y\,d\H^{n-1}=H_0\,\int_\Sigma \langle\nu_\Om,Y\rangle\,d\H^{n-1}\,;
  \end{equation}
  \end{enumerate}
\medskip

\noindent then, for some $t_0\in(0,\bar{r})$, $\Omega=N\times(0,t_0)$ and $H_0=(n-1)\,h'(t_0)/h(t_0)>0$.
\end{theorem}

\begin{remark}[On the relation between Theorem \ref{thm hk inq} and Theorem \ref{thm rigidity}]\label{remark on the relation}
  {\rm As detailed in the proof of Theorem \ref{thm rigidity}, by testing the constant mean curvature condition \eqref{vale div thm}  with the vector field $h\,(\pa/\pa r)$, we see that $H_0$ appearing in \eqref{vale div thm} is positive (so that ${\rm (A2)}$' implies ${\rm (A2)}$), and that either \eqref{conclusion a} or \eqref{conclusion b} (depending on whether $\pa\Om=\ov\Sigma$ or $\pa\Om=\ov\Sigma\cup N_0$) {\it holds as an identity}. Moreover, by Theorem \ref{thm white},  ${\rm (A3)}$' implies the validity of ${\rm (A3)}$. Therefore, under the assumptions of Theorem \ref{thm rigidity}, conclusion (c) of Theorem \ref{thm hk inq} holds too. When ${\rm (H3)^*}$ holds, then \eqref{for H3 star} immediately implies that $\nu_\Om(p)$ is parallel to $(\pa/\pa r)|_p$ at every $p\in\Sigma$: this information, combined with standard facts on sets of finite perimeter and with the positivity of $H_\Sigma$, immediately implies that $\Om$ is bounded by a single slice. When, instead, only ${\rm (H3)}$ is assumed, the information in \eqref{for H3 star} may be trivial. In this second case, arguing as in \cite{brendle}, we deduce from umbilicality and the Codazzi equations that $\nu_\Om(p)$ is an eigenvector of $({\rm Ric}_M)|_p$ at every $p\in\Sigma$. Since ${\rm (H4)}$ implies that $(\pa/\pa r)|_p$ is a simple eigenvector of $({\rm Ric}_M)|_p$, we thus find that, at each $p\in\Sigma$, $\nu_\Om(p)$ is either parallel {\it or orthogonal} to $(\pa/\pa r)|_p$. Concluding rigidity from this weaker information using only standard facts on sets of finite perimeter does not seem immediate; however, the fact that $M\setminus\Om$ has positive reach in $(M,g^*)$ can be exploited to quickly reach the desired conclusion.}
\end{remark}

\begin{proof}[Proof of Theorem \ref{thm hk inq}] {\it Preparation of $M$}: The results of Section \ref{section positive reach} and Section \ref{section viscous mc} require the completeness of the ambient manifold. Notice that $(M,g)$ is not complete. A first problem is that geodesics in $(M,g)$ may arrive in finite time to the horizon $ N_0$: this issue is fixed by passing from $g$ to $g^{*}=f^{-2}\,g$. A different issue, however, is the behavior of geodesics near the $\bar{r}$-end of $M$. To fix this second problem we argue as follows. By assumption $\overline{\Sigma}\subset M^\circ$, there is $(a,b)\cc(0,\bar{r})$ such that
\begin{equation}
  \label{ab}
  \ov\Sigma\subset N\times\Big[a,\frac{a+b}2\Big]\,,\qquad \Om\subset N\times \Big(0,\frac{a+b}2\Big)\,.
\end{equation}
Correspondingly we can consider a smooth positive function $h_b:[0, +\infty)\to\R$ so that $h_b=h$ on $[0,b]$, $h_b'>0$ on $(0,\infty)$, and
\[
\sup_{[0,\infty)}h_b'<\infty\,,\qquad \int_b^\infty \frac{dt}{h_b'(t)}=+\infty\,.
\]
We introduce the metrics $g_b=dr \otimes dr+h_b(r)^2g_N$ and $g^{*}_b=f_b^{-2}\,g_b$ on $N\times (0,\infty)$, where $f_b(p)=h_b'(r(p))$. Since ${\rm (H1)}$ and $h_b=h$ on $[0,b]$ imply
\[
\int_{0}^b \frac{dt}{h_b'(t)}=+\infty\,,
\]
we easily see that $g^{*}_b$-geodesic balls centered at points $p\in N\times(0,+\infty)$ are contained in compact slabs of the form $N\times[s,t]$ ($0<s<t<\infty$). In particular, by the Hopf--Rinow theorem,  $(N\times(0,\infty),g^{*}_b)$ is a {\it complete} Riemannian manifold. Notice that $h_b$ satisfies assumption ${\rm (H3)}$ on $(0,b)$, where it coincides with $h$, but, possibly, not on $(0,\infty)$.

\medskip

We have thus reduced to the following situation: $M=N\times(0,\infty)$; the metric $g=dr\otimes dr+h(r)^2\,g_N$ is such that ${\rm (H0)}$ and ${\rm (H1)}$ hold, ${\rm (H2)}$ holds on $(0,\infty)$ (i.e., $h'>0$ on $(0,\infty)$), ${\rm (H3)}$ holds on $(0,b)$, and
\begin{equation}
  \label{f limitata}
  \sup_{M}f<\infty\,;
\end{equation}
the metric $g^*=f^{-2}\,g$ is such that $(M,g^*)$ is a complete Riemannian manifold; and, finally, $\Sigma$ and $\Omega$ satisfy \eqref{ab} in addition to assumptions ${\rm (A1)}$, ${\rm (A2)}$, and ${\rm (A3)}$.

\medskip

\noindent {\it The ``vertical'' vector field $X$}: The shortest path between points $(x,t_1)$ and $(x,t_2)$ in $(M,g^*)$ with $t_1\in(0,t_2)$ is given by $s\in[t_1,t_2]\mapsto (x,s)$ (while $\dist_{g^*}((x,0),(x,t))=+\infty$ for every $x\in N$ and $t>0$). ``Vertical'' segments are thus length minimizing geodesics in $(M,g^*)$, and the vector field $\pa/\pa r$ has a special role in the geometry of $(M,g^*)$. It is also convenient to consider, alongside with $\pa/\pa r$, its rescaled version
\[
X=h\,\frac{\pa}{\pa r}\,.
\]
Simple computations show that
\begin{eqnarray}
  \label{div ddr}
  &&\Div(\pa/\pa r)=(n-1)\,\frac{f}{h\circ r}\,,
  \\
  \label{divergence of pa par tangential}
  &&\Div^\Gamma(\pa/\pa r)=\Big\{(n-1)-\sum_{i=1}^n\,\langle\s_i,\pa_n\rangle_g^2\Big\}\,\frac{f}{h\circ r}\,,\qquad\mbox{on $\Gamma$}\,,
  \\
  \label{Div Gamma X}
  &&\Div\,X=n\,f\quad\mbox{on $M$}\,,\qquad \Div^\Gamma X=(n-1)\,f\quad\mbox{on $\Gamma$}\,,
\end{eqnarray}
whenever $\Gamma$ is a $C^1$-hypersurface and $\{\s_i\}_{i=1}^{n-1}$ denotes a $g$-orthonormal basis of $T_p\Gamma$ for some $p\in\Gamma$. An immediate consequence of \eqref{divergence of pa par tangential} is that if $\Gamma$ is a {\it closed} $C^{1,1}$-hypersurface in $M^\circ$ with $\H^{n-1}(\Gamma)<\infty$ and with mean curvature vector $\vec{H}_\Gamma$ in $(M,g)$, and if $h''\circ r\ge0$ on $\Gamma$, then
\begin{equation}\label{lem: integration by parts eq2}
(n-1)\,\H^{n-1}(\Gamma)\ge \int_{\Gamma} \frac{\langle \vec{H}_\Gamma ,X\rangle_g}f\, d\H^{n-1}\,.
\end{equation}
Indeed, by \eqref{Div Gamma X} we find
\[
\Div^\Gamma\Big(\frac{X}f\Big)=(n-1)-\frac{\langle \nabla^\Gamma f,X\rangle_g}{f^2}\,,
\]
where $\nabla f=[(h''/h)\circ r]\,X$. Denoting by $X^\Gamma$ the projection of $X$ along $T\Gamma$, we find $\langle X,\nabla^\Gamma f\rangle_g=[(h''/h')\circ r]\,|X^\Gamma|_g^2$. Hence,  by applying the divergence theorem to $X/f$ on $\Gamma$ and by using $h''\circ r\ge0$ on $\Gamma$, we find \eqref{lem: integration by parts eq2}.


\medskip

\noindent  {\it ``Immersed'' geodesic flow from $\Sigma$}: The vector field
\[
\nu^*_\Om=f\,\nu_\Om
\]
is a smooth $g^*$-unit normal to $\Sigma$, pointing out of $\Om$. Now, denoting by $\exp^* $ the exponential map in the complete Riemannian manifold $(M,g^{*})$, we can define a smooth map $\Phi:\Sigma\times(0,\infty)\to M$ by setting
\[
\Phi(x,t)=\Phi_t(x)=\exp^*(x, -t\,\nu^*_\Om(x))\,,\qquad x\in\Sigma\,,t>0\,.
\]
In this way, denoting by $J^\Sigma\Phi_t$ the tangential Jacobian\footnote{Here and in the following, $\H^k$ and $J$ always denote Hausdorff measures and Jacobians computed with respect to the metric $g$.} of $\Phi_t$ along $\Sigma$,
\begin{equation}
  \label{Phi basics}
\Phi(x,0)=x\,,\quad \frac{\pa\Phi}{\pa t}(x,0)=-f(x)\,\nu_\Om(x)\,,\quad J^\Sigma\Phi_0(x)=1\,,\qquad\forall x\in\Sigma\,.
\end{equation}
By \eqref{ab} and \eqref{Phi basics}, we have that $\Phi_t(x)\in N\times(0,b)$ with $J^\Sigma\Phi_t(x)>0$ for every $t$ small enough. We can thus define a lower semicontinuous, positive function $R_\Sigma:\Sigma\to (0,\infty]$ by setting
\begin{eqnarray*}
R_\Sigma(x)\!\!&=&\!\!\min\Big\{\inf\big\{t>0: J^\Sigma\Phi_t(x)=0\big\},
\\
&&\hspace{2cm}\inf\big\{t>0:\Phi_t(x)\not\in N\times(0,b)\big\}\Big\}\,,\qquad x\in\Sigma\,,
\end{eqnarray*}
so to have
\begin{eqnarray}  \label{effetto Rsigma su Asigma}
  &&\Phi_t(x)\in N\times(0,b)\quad\mbox{and}\quad J^\Sigma\Phi_t(x)>0\,,
  \\\nonumber
  &&\forall (x,t)\in A_\Sigma:=\Big\{(x,t):x\in\Sigma\,,t\in(0,R_\Sigma(x))\Big\}\,.
\end{eqnarray}
By the Gauss lemma (see, e.g., \cite[pag.\ 60]{MR1390760}), for every $(x,t)\in A_\Sigma$,
\begin{equation}
  \label{speed of pa phi pa t}
  \frac{\pa\Phi}{\pa t}(x,t)\in\Big(d\Phi_t(x)[T_x\Sigma]\Big)^\perp\,,\qquad \Big|\frac{\pa\Phi}{\pa t}(x,t)\Big|_g=f(\Phi_t(x))\,.
\end{equation}
In particular, the tangential Jacobian $J^{A_\Sigma}\Phi$ of $\Phi$ along $A_\Sigma$ is related to $J^\Sigma\Phi_t$ by the identity
\begin{equation}\label{prop: hk and n,h sets: eq2}
	(J^{A_\Sigma}\Phi)(x,t) = f(\Phi_t(x))\,J^\Sigma\Phi_t(x)\,,\qquad\forall (x,t) \in A_\Sigma\,.
\end{equation}
We now notice that, for every $t\in(0,\infty)$, $\{R_\Sigma>t\}$ is an open subset of $\Sigma$, and
\[
\Gamma_t=\Phi_t\big(\{R_\Sigma>t\}\big)\,,\qquad t>0\,,
\]
is a smooth {\it immersed} hypersurface in $M$. Indeed, by construction, for every $(x,t)\in A_\Sigma$ there is an open neighborhood $W$ of $x$ in $\Sigma$ such that $(\Phi_t)|_W$ is a smooth embedding. Correspondingly, we denote\footnote{Notice carefully that $\Phi$ may not be injective on the whole $A_\Sigma$, therefore we will not be able to consider $H$ and ${\rm II}$ as functions on $\Phi(A_\Sigma)\subset M$.} by $H(x,t)$ and ${\rm II}(x,t)$ the scalar mean curvature and the second fundamental form (in the metric $g$) of the smooth hypersurface $\Phi_t(W)$ at the point $\Phi_t(x)$ and with respect to the normal 
	\begin{equation}\label{def: nu immersed}
		\nu(x,t)=-\frac{1}{f(\Phi_t(x))}\frac{\partial \Phi}{\partial t}(x,t),
	\end{equation}
(see \eqref{prop: hk and n,h sets: eq2}). Notice that $ \nu(x,0) = \nu_\Omega(x) $ for $ x \in \Sigma $.
 Since, by \eqref{effetto Rsigma su Asigma}, $\Phi_t(x)\in N\times(0,b)$ for every $t\in(0,R_\Sigma(x))$, and since $h$ satisfies ${\rm (H3)}$ on $(0,b)$, the pointwise calculations in \cite[Proposition 3.2]{brendle} (which are based on ${\rm (H0)}$ and ${\rm (H3)}$ and the Riccati equation) can be repeated {\it verbatim} to show that, everywhere on $A_\Sigma$,
\begin{eqnarray}
\label{screaming trees}
&&\frac{\partial}{\partial t}\Big( \frac{H}{f\circ\Phi}\Big) \ge|{\rm II}|^2
 \geq \frac{H^2}{n-1}\,,
\\
\label{prop: hk and n,h sets: eq5}
&&\frac{\partial}{\partial t}\Big( \frac{f\circ\Phi}{H}\Big) \leq - |{\rm II}|^2\,\frac{(f\circ \Phi)^2}{H^2}\,
 \leq -\frac{(f \circ \Phi)^2}{n-1}\,.
\end{eqnarray}
Since $H(\cdot,0)>0$ on $\Sigma$ by assumption ${\rm (A2)}$, we see that \eqref{screaming trees} implies
\begin{equation}
  \label{i nearly}
  \mbox{$H$ is positive on $A_\Sigma$}\,.
\end{equation}
Moreover, we have that
\begin{equation}\label{prop: hk and n,h sets: eq6}
\frac{\partial}{\partial t}\,J^\Sigma\Phi_t(x)= -[(f\circ\Phi)\,H](x,t) \, J^\Sigma \Phi_t(x)\,,\qquad\forall (x,t)\in A_\Sigma\,.
\end{equation}
Indeed, given $(x,t)\in A_\Sigma$ and $W$ as above, if $W'$ is an arbitrary open subset of $W\subset\Sigma$ then, by the area formula,
\begin{eqnarray*}
  \frac{d}{ds}\Big|_{s=t}\H^{n-1}(\Phi_s(W'))=  \frac{d}{ds}\Big|_{s=t}\,\int_{W'}J^\Sigma\Phi_s\,d\H^{n-1}
  =
  \int_{W'}\frac{\pa}{\pa t}J^\Sigma\Phi_t\,d\H^{n-1}
\end{eqnarray*}
while, by the formula for the first variation of the area and by \eqref{speed of pa phi pa t}
\begin{eqnarray*}
\frac{d}{ds}\Big|_{s=t}\H^{n-1}(\Phi_s(W'))\!\!\!&=&\!\!\!-\int_{W'} [(f\circ\Phi)\,H](y,t)\,J^\Sigma\Phi_t(y)\,d\H^{n-1}_y\,,
\end{eqnarray*}
so that \eqref{prop: hk and n,h sets: eq6} follows by arbitrariness of $W'$. We finally notice that
\begin{eqnarray}
  \nonumber
  &&\mbox{if $x\in\Sigma$, $R_\Sigma(x)<\infty$, and $J^\Sigma\Phi_t(x)\to 0$ as $t\to R_\Sigma(x)^-$}\,,
  \\
  \label{infinito}
  &&\mbox{then $H(x,t)\to+\infty$ as $t\to R_\Sigma(x)^-$}\,.
\end{eqnarray}
Indeed, \eqref{prop: hk and n,h sets: eq6} gives
\[
\log\big(J^\Sigma\Phi_t(x)\big)=-\int_0^t\,[(f\circ\Phi)\,H](x,s)\,ds\,,\qquad\forall t\in(0,R_\Sigma(x))\,.
\]

\medskip

\noindent {\it A refinement of \eqref{prop: hk and n,h sets: eq5}}: We claim that, everywhere on $A_\Sigma$,
\begin{eqnarray}\label{T careful}
  -\Big\{\frac{\pa}{\pa t}\Big( \frac{f\circ\Phi}{H}\Big)+\frac{(f\circ\Phi)^2}{n-1}\Big\}
  &=&(f\circ\Phi)^2\,\Big\{\frac{ |{\rm II}|^2}{H^2}-\frac1{n-1}\Big\}
  \\\nonumber
  &&+\frac{f\circ\Phi}{H^2}\,h'\,\big\{{\rm Ric}_N-\rho\,(n-2)\,g_N\big\}(\nu,\nu)
  \\\nonumber
  &&+\frac{f\circ\Phi}{H^2}\,\frac{h^3}2\,\big(\mathcal{M}[h]\big)'\,g_N(\nu,\nu)\,,
\end{eqnarray}
where, by definition,
\begin{equation}
  \label{def of M h}
  \mathcal{M}[h]=2\,\frac{h''}h-(n-2)\,\frac{\rho-(h')^2}{h^2}\,.
\end{equation}
Indeed, setting $T=(\Delta f)\,g-\Der^2\,f+f\,{\rm Ric}_M$, by \cite[Proposition 2.1]{brendle} we have that
\begin{eqnarray}\label{T1}
T= h'\,\big\{{\rm Ric}_N-\rho\,(n-2)\,g_N\big\}+\frac{h^3}2\,\big(\mathcal{M}[h]\big)'\,g_N\,,
\end{eqnarray}
while the computations in \cite[Proposition 3.2]{brendle} give
\begin{equation}
  \label{T2}
  \frac{\pa}{\pa t}\Big( \frac{f\circ\Phi}{H}\Big)=
-(f\circ\Phi)^2\,\frac{ |{\rm II}|^2}{H^2}-\frac{(f\circ\Phi)}{H^2}\,T(\nu,\nu)\,.
\end{equation}
The combination of \eqref{T1} and \eqref{T2} leads immediately to \eqref{T careful}.

\medskip

\noindent  {\it Geodesic flow from $\Sigma$}: We now consider the (embedded) geodesic flow of $\Sigma$ (which is the main structure used in Brendle's argument), i.e. we relate $\Phi$ to the distance function from $\Sigma$ in $(M,g^*)$. Let us define, for the sake of brevity, $u^*_\Sigma:\overline{\Omega}\to[0,\infty)$ and $R_\Sigma^*:\Sigma\to(0,\infty)$ by setting
\begin{eqnarray*}
u^*_\Sigma(p)& =& \dist_{g^{*}}(p,\Sigma)\,,\qquad \qquad\hspace{2.2cm} p\in\ov{\Om}\,,
\\
R^*_\Sigma(x)&=&\rho^*_{\ov\Sigma}(x,-\nu^*_\Om(x))
\\
&=&\sup\big\{s>0:s=u_\Sigma^*(\Phi_s(x))\big\}\,,  \qquad x\in\Sigma\,,
\end{eqnarray*}
(where, for $\Gamma$ closed in $M$, $\rho^*_\Gamma$ is defined as in \eqref{def rho sigma} with respect to the metric $g^*$)  and then consider the sets
\begin{eqnarray*}
\Omega_t&=& \big\{p\in\Om:u_{\Sigma}^*(p)>t\big\}\,,
\\
\Sigma_t&=&\big\{p \in \Omega: u^*_\Sigma(p)=t\big\}=M^\circ\cap\pa\Om_t\,,
\\
\Sigma^\ast_t&=&\Phi_t(\{R^*_\Sigma>t\})\subset\Sigma_t\,,
\\
A_\Sigma^*&=&\big\{(x,t) : x \in \Sigma,\, t\in(0,R^*_\Sigma(x))\big\}\,.
\end{eqnarray*}
(Notice that if $\pa\Om=\ov\Sigma$, then $\pa\Om_t=\Sigma_t$; if, otherwise, $\pa\Om=\ov\Sigma\cup N_0$, then $\pa\Om_t=\Sigma_t\cup N_0$.)  It is easily seen that $R_\Sigma^*$ is continuous on $\Sigma$, so that $\{R^*_\Sigma>t\}$ is an open subset of $\Sigma$ for every $t>0$, and $A_\Sigma^*$ is open. The fact that $\Phi$ is a diffeomorphism on $A_\Sigma^*$ with values in $\Om$ is standard (since $\Sigma$ is smooth), so that
\begin{equation}
  \label{RSigma larger than RSigmastar}
  \Phi(A_\Sigma^*)\subset\Omega\,,\qquad
  R_\Sigma(x)\ge R_\Sigma^*(x)\quad\forall x\in\Sigma\,,\qquad A_\Sigma^*\subset A_\Sigma\,,
\end{equation}
and $\Phi_t$ is a smooth embedding of $\{R_\Sigma^*>t\}$ into $M$. In particular, for each $t>0$, $\Sigma^\ast_t$ is a (possibly empty, embedded) hypersurface in $M$. (Notice that $\Gamma_t$ is, in general, {\it larger} than $\Sigma_t^*$, immersed but not embedded, and {\it unrelated} to  $u_\Sigma^*$.) The vector field  (see \eqref{def: nu immersed})
\begin{equation}
  \label{def of nu}
  \nu_t(y)=\nu(\Phi_t^{-1}(y),t)\,,\qquad y\in\Sigma_t^*\,,
\end{equation}
is a unit normal vector field to $\Sigma_t^*$ in $(M,g)$ with the property that
\begin{equation}
  \label{def of Hstart}
  H(\Phi_t^{-1}(y),t)=H_{\Sigma_t^*}(y)\qquad \mbox{$y \in  \Sigma_t^*$}\,,
\end{equation}
where $H_{\Sigma_t^*}$ is the scalar mean curvature of $\Sigma_t^*$ with respect to $\nu_t$. We now prove three important geometric properties of the family $\{\Sigma_t^*\}_t$, namely, we show that
\begin{eqnarray}
\label{prop: hk and n,h sets: eq1}
&&\H^n(\Omega \setminus \Phi(A_\Sigma^*))=0\,,
\\
\label{prop: hk and n,h sets: eq13}
&&\H^{n-1}(\Sigma_t \setminus \Sigma^\ast_t) =0\,,\qquad\mbox{for $\mathcal{L}^1$-a.e.\ $t>0$}\,,
\end{eqnarray}
and that, when $\pa\Om=\Sigma\cup N_0$,
\begin{eqnarray}
\label{lsigma 3}
\H^{n-1}(\Sigma_t^*)\ge h(0)^{n-1}\,\vol(N)\,,\qquad\mbox{for $\L^1$-a.e. $t>0$}\,.
\end{eqnarray}
We begin noticing that \eqref{prop: hk and n,h sets: eq1} is immediate to prove when $\Sigma$ is a \emph{closed} smooth  hypersurface, since, in that case, we trivially see that  $\Omega \setminus \Phi(A_\Sigma^*) \subseteq {\rm Cut}^\ast(\Sigma)$, where ${\rm Cut}^*$ denotes the cut-locus in $(M,g^*)$, and $ \mathcal{H}^n({\rm Cut}^\ast(\Sigma))=0 $ by Theorem \ref{thm normal bundles closed sets}-(v). In our non-smooth setting, we begin noticing that,  by assumption ${\rm (A3)}$ and Lemma \ref{lem: mh sets and conformal metric}, there is $\l^*>0$ such that $\overline{\Sigma}$ is a White $(n-1,\lambda^*)$-subset of $(M,g^{*})$. Now, by construction,
\[
\Omega \setminus  \Phi(A_\Sigma^*) \subseteq {\rm Cut}^*(\overline{\Sigma}) \cup \exp\big[ \mathcal{N}(\overline{\Sigma}) \llcorner(\overline{\Sigma} \setminus \Sigma)  \big]\,,
\]
Since $\overline{\Sigma} $ is a White $(n-1, \lambda^*)$-set in $(M,g^*)$ and $ \H^{n-1}(\overline{\Sigma} \setminus \Sigma) =0 $ we conclude from Theorem \ref{thm Santilli riem} that
\[
\H^n(\mathcal{N}(\overline{\Sigma}) \llcorner(\overline{\Sigma} \setminus \Sigma)) =0 \quad \textrm{and} \quad \H^n\big(\exp^*\big[\mathcal{N}(\overline{\Sigma}) \llcorner(\overline{\Sigma} \setminus \Sigma)	\big]\big) =0\,.
\]
Since $ \H^n({\rm Cut}^\ast(\overline{\Sigma})) =0 $ , we conclude the proof of \eqref{prop: hk and n,h sets: eq1}. By \eqref{prop: hk and n,h sets: eq1} and the coarea formula we have
\[
0=\int_{\Omega \setminus \Phi(A_\Sigma^*)}|\nabla u^*_\Sigma|_g\, d\H^n
=\int_{0}^\infty \H^{n-1}(\Sigma_s \setminus \Phi(A_\Sigma^*))\, ds\,,
\]
which immediately implies \eqref{prop: hk and n,h sets: eq13} by $\Sigma_s^*=\Sigma_s\cap\Phi(A_\Sigma^*)$. Finally, to prove \eqref{lsigma 3}, we notice that, since $u_\Sigma^*$ is a Lipschitz function, $\Om_t=\{u_\Sigma^*>t\}$ is a set of finite perimeter in $M$ for $\L^1$-a.e. $t>0$. Since $\pa\Om=\Sigma\cup N_0$ implies $\pa\Om_t= N_0\cup\Sigma_t$, by \eqref{prop: hk and n,h sets: eq13} the reduced boundary $\pa^*\Om_t$ of $\Om_t$ is $\H^{n-1}$-equivalent to the union of $N_0$ and $\Sigma_t^*$, with measure theoretic outer $g$-unit normal $\nu_{\Om_t}$ such that $\nu_{\Om_t}=-\pa/\pa r$ on $N_0$ and  $\nu_{\Om_t}=\nu_t$ on $\Sigma_t^*$; in particular,
\[
\int_{\Om_t}\Div(\pa/\pa r) = \int_{\Sigma_t^*}\langle \nu_t,\pa/\pa r\rangle_g\,d\H^{n-1}-\H^{n-1}(N_0)\,.
\]
By \eqref{div ddr}, and since both $\pa/\pa r$ and $\nu_t$ have unit length in $g$, we deduce $\H^{n-1}(\Sigma_t^*)\ge \H^{n-1}(N_0)$, which is \eqref{lsigma 3}.

\medskip

\noindent {\it A general Heintze--Karcher inequality}: We now prove a general Heintze--Karcher inequality, see \eqref{general hk} below, which implies both \eqref{conclusion a} and \eqref{conclusion b}, and which allows one to deduce the crucial positive reach information contained in conclusion (c) when equality holds in either  \eqref{conclusion a} or \eqref{conclusion b}. We start noticing that, by \eqref{prop: hk and n,h sets: eq5} and \eqref{prop: hk and n,h sets: eq6},
\begin{eqnarray}\label{dai}
  &&\int_\Sigma\,d\H^{n-1}_x\int_0^{R_\Sigma(x)}(f\circ \Phi)^2(x,t)\,J^\Sigma\Phi_t(x)\,dt
  \\\nonumber
&&  \hspace{1cm}\le
  -(n-1)\,\int_\Sigma\,d\H^{n-1}_x\int_0^{R_\Sigma(x)}\frac{\pa}{\pa t}\Big(\frac{f\circ \Phi}{H}\Big)\,(x,t)\,J^\Sigma\Phi_t(x)\,dt
  \\\nonumber
  &&  \hspace{1cm}=-(n-1)
  \int_\Sigma\,d\H^{n-1}_x\int_0^{R_\Sigma(x)}(f\circ \Phi)^2\,(x,t)\,J^\Sigma\Phi_t(x)\,dt
  \\\nonumber
  &&\hspace{2cm}-(n-1)\,\int_\Sigma\Big[\Big(\frac{f\circ \Phi}{H}\Big)\,(x,t)\,J^\Sigma\Phi_t(x)\Big]\Big|_{t=0}^{t=R_\Sigma(x)}d\H^{n-1}_x\,.
\end{eqnarray}
By \eqref{prop: hk and n,h sets: eq2}, the area formula, and \eqref{prop: hk and n,h sets: eq1} we obtain
\begin{eqnarray*}
&&\int_\Sigma\,d\H^{n-1}_x\int_0^{R_\Sigma^*(x)}(f\circ \Phi)^2\,(x,t)\,J^\Sigma\Phi_t(x)\,dt
\\
&&=\int_{A_\Sigma^*}(f\circ \Phi)\,J^{A_\Sigma}\Phi\,d\H^n=\int_{\Phi(A_\Sigma^*)}f\,d\H^n=\int_\Om f\,d\H^n\,,
\end{eqnarray*}
while by \eqref{Phi basics},
\[
\Big[\int_\Sigma\Big(\frac{f\circ \Phi}{H}\Big)\,(x,t)\,J^\Sigma\Phi_t(x)d\H^{n-1}_x\Big]\Big|_{t=0}=\int_\Sigma\,\frac{f}{H_\Sigma}\,d\H^{n-1}\,,
\]
so that \eqref{dai} gives
\begin{eqnarray}\nonumber
  &&n\,\int_\Om f\,d\H^n+\int_\Sigma\,d\H^{n-1}_x\int_{R_\Sigma^*(x)}^{R_\Sigma(x)}(f\circ \Phi)^2\,(x,t)\,J^\Sigma\Phi_t(x)\,dt
  +\L(\Sigma)
  \\\label{general hk}
  &&\le (n-1)\,\int_\Sigma\frac{f}{H_\Sigma}\,d\H^{n-1}\,,
\end{eqnarray}
where
\[
\L(\Sigma)=(n-1)\,\int_{\Sigma}\,\Big[\lim_{t\to R_\Sigma(x)^-}\,\Big(\frac{f\circ \Phi}{H}\Big)\,(x,t)\,J^\Sigma\Phi_t(x)\Big]\,d\H^{n-1}_x\,.
\]
Notice that, for every $x\in\Sigma$, $t\mapsto J^\Sigma\Phi_t(x)\,[(f\circ\Phi)/H](x,t)$ is decreasing on $(0,R_\Sigma(x))$ thanks to \eqref{prop: hk and n,h sets: eq5} and \eqref{prop: hk and n,h sets: eq6}: in particular, the integrand in the definition of $\L(\Sigma)$ is a well-defined non-negative function, \eqref{conclusion a} follows immediately from \eqref{general hk}, and conclusion (a) is proved.

\medskip

\noindent {\it Conditional proof of conclusions (b) and (c)}: We now prove conclusions (b) and (c) assuming the validity of the following inequality:
\begin{equation}
  \label{claim}
  \L(\Sigma)\ge h(0)^n\,{\rm vol}(N)\,,\qquad\mbox{when $\pa\Om=\Sigma\cup N_0$}\,.
\end{equation}
Indeed, if \eqref{claim} holds, then \eqref{general hk} definitely implies \eqref{conclusion b}, that is is conclusion (b). Moreover, if equality holds in either \eqref{conclusion a} or \eqref{conclusion b}, then inequality \eqref{dai} (appearing in the derivation of \eqref{general hk}) must hold as an identity. Therefore, since \eqref{prop: hk and n,h sets: eq5} was used in proving \eqref{dai}, we find that if equality holds in either \eqref{conclusion a} or \eqref{conclusion b}, then
\begin{eqnarray}
  \label{condition one}
  &&
  \frac{\partial}{\partial t}\Big( \frac{f\circ\Phi}{H}\Big)=-\frac{(f \circ \Phi)^2}{n-1}\,,\qquad\mbox{on $A_\Sigma$}\,,
  \\
  \label{condition two}
  &&\int_\Sigma\,d\H^{n-1}_x\int_{R_\Sigma^*(x)}^{R_\Sigma(x)}(f\circ \Phi)^2\,(x,t)\,J^\Sigma\Phi_t(x)\,dt=0\,.
\end{eqnarray}
By \eqref{T careful}, we see that \eqref{condition one} gives
\begin{equation}\label{umbilicality condition}
	|{\rm II}|^2=\frac{H^2}{n-1}\,,\qquad\mbox{on $A_\Sigma$}\,,
\end{equation}
which, tested with $t=0$, implies that $\Sigma$ is umbilical in $(M,g)$ (the first part of conclusion (c)), as well as
\[
\frac{h^3}2\,\big(\mathcal{M}[h]\big)'\,g_N(\nu,\nu)=0\,,\qquad\mbox{on $A_\Sigma$}\,,
\]
which, tested with $t=0$, implies the validity of \eqref{for H3 star}. A more delicate argument is needed to deduce from \eqref{condition two} that $M\setminus\Om$ has positive reach in $(M,g^*)$ (the second part of conclusion (c)), and it goes as follows: Since $f>0$ on $M$ (by assumption (${\rm H2}$)) and $J^\Sigma\Phi_t(x)>0$ for every $t\in(0,R_\Sigma(x))$ (by definition of $R_\Sigma(x)$), \eqref{condition two} implies that
\begin{equation*}
  R_\Sigma(x)= R_\Sigma^*(x)\,,\qquad\mbox{for $\H^{n-1}$-a.e. $x\in\Sigma$}\,,
\end{equation*}
 whence we infer from the lower semicontinuity of $ R_\Sigma $ and the continuity of $ R^\ast_\Sigma $ that
\begin{equation}\label{Rs are equal}
	R_\Sigma(x)= R_\Sigma^*(x)\,,\qquad\mbox{for every $x\in\Sigma$}\,.
\end{equation}
Let $x \in \Sigma $ be such that  $R_\Sigma(x)<\infty$. Should it be that
\begin{equation}
  \label{cannot}
  \dist_{g^*}(\Phi_t(x),N_0\cup N_b)\to 0^+\qquad\mbox{as $t\to R_\Sigma(x)^-$}\,,
\end{equation}
the facts that $R_\Sigma(x)<\infty$ and $\dist_{g^*}(N_0,\Sigma)\ge\dist_{g^*}(N_0,N_a)=+\infty$ would imply
\[
\dist_{g^*}(\Phi_t(x),N_b)\to 0^+\qquad\mbox{as $t\to R_\Sigma(x)^-$}\,,
\]
and thus, by \eqref{Rs are equal}, $\Phi(A_\Sigma^*)\subset\Om$, and \eqref{ab}, that
\[
\lim_{t\to R_\Sigma(x)^-}\Phi_t(x)\in\ov\Om\cap N_b=\varnothing\,,
\]
a contradiction. Since \eqref{cannot} cannot occur, $R_\Sigma(x)<\infty$ must then imply $J^\Sigma\Phi_t(x)\to 0^+$ as $t\to R_\Sigma(x)^-$: in this case, \eqref{infinito} holds, and we can integrate \eqref{condition one} over $t\in(0,R_\Sigma(x))$ and take advantage of \eqref{Phi basics} so to find
\[
-\frac{f(x)}{H_\Sigma(x)}=-\int_0^{R_\Sigma(x)}\frac{f(\Phi_t(x))^2}{n-1}\,dt\ge-\frac{R_\Sigma(x)}{n-1}\,\sup_M\,f^2\,,
\]
where $\sup_M f^2<\infty$ thanks to \eqref{f limitata}. Since $0<H_\Sigma\le\l$ on $\Sigma$ by assumption ${\rm (A3)}$, we conclude (using again \eqref{ab}) that
\[
R_\Sigma(x)\ge \frac{(n-1)}{\l\,\sup_M f^2}\,\inf_\Sigma f\ge \frac{(n-1)}{\l\,\sup_M f^2}\,\inf_{[a,b]}h'\,.
\]
We have thus proved the existence of a positive constant $c(\Sigma)$ such that
\begin{equation}
  \label{f}
  R_\Sigma^*(x)=R_\Sigma(x)\ge c(\Sigma)\,,\qquad\forall x\in\Sigma\,.
\end{equation}
By assumption ${\rm (A2)}$, $\pa(M\setminus\Om)=\ov{\Sigma}$, therefore
\begin{eqnarray}
\label{fib1}
\mathcal{N}^1(M\setminus\Om)\!\!&\subset&\!\!\mathcal{N}^1\ov\Sigma\,,
\\
\label{fib2}
\mathcal{N}^1(M\setminus\Om)\cap\big(\mathcal{N}^1\ov\Sigma\,\,\llcorner\Sigma\big)\!\!&=&\!\!\big\{(x,-\nu_\Om^*(x)):x\in\Sigma\big\}\,.
\end{eqnarray}
At the same time, by applying Theorem \ref{thm Santilli riem} with $\Gamma=\ov{\Sigma}$ (which is admissible by assumption ${\rm (A3)}$), $m=n-1$, and $Z=\ov\Sigma\setminus\Sigma$ (which is admissible since, by assumption ${\rm (A1)}$, $\H^{n-1}(\ov\Sigma\setminus\S)=0$), we find that
\begin{equation}
\label{fib3}
\H^{n-1}\big(\mathcal{N}^1\ov\Sigma\,\,\llcorner\,\big(\ov\Sigma\setminus\Sigma\big)\big)=0\,.
\end{equation}
By combining \eqref{fib1}, \eqref{fib2} and \eqref{fib3} we thus find
\[
(x,\eta)=(x,-\nu_\Om^*(x))\qquad\mbox{for $\H^{n-1}$-a.e. $(x,\eta)\in \mathcal{N}^1(M\setminus\Om)$}\,,
\]
so that the function, for $\H^{n-1}$-a.e. $(x,\eta)\in \mathcal{N}^1(M\setminus\Om)$ we have
\begin{eqnarray*}
  \rho^*_{M\setminus\Om}(x,\eta)&=&\rho^*_{M\setminus\Om}(x,-\nu_\Om^*(x))
  \\
  &=&\sup\big\{s>0:s=\dist_{g^*}\big(\exp^*\big(x,-s\,\nu_\Om^*(x)\big),M\setminus\Om\big)\big\}\,.
\end{eqnarray*}
Since $\dist_{g^*}(p,M\setminus\Om)=\dist_{g^*}(p,\ov\Sigma)$ for $p\in\Om$ and since
\[
\exp^*(x,-s\,\nu_\Om^*(x))\in\Om\,,\qquad\forall s\in(0,\rho^*_{M\setminus\Om}(x,-\nu_\Om^*(x)))\,,
\]
by \eqref{f} we conclude that, for $\H^{n-1}$-a.e. $(x,\eta)\in \mathcal{N}^1(M\setminus\Om)$,
\[
\rho^*_{M\setminus\Om}(x,\eta)=\rho^*_{\ov\Sigma}(x,-\nu_\Om^*(x))=R_\Sigma^*(x)\ge c(\Sigma)>0\,.
\]
We can thus apply Theorem \ref{thm normal bundles closed sets}-(vii) and Remark \ref{remark pos reach} to finally conclude that $M\setminus\Omega $ is a set of  positive reach in $(M,g^*)$. We are thus left to prove \eqref{claim} to complete the proof of conclusions (b) and (c).

\medskip

\noindent {\it Proof of \eqref{claim}}: We start by setting, for every $(x,t)\in\Sigma\times[0,\infty)$,
\[
g(x,t)=\left\{
\begin{split}
  &\Big(\frac{f\circ \Phi}{H}\Big)\,(x,t)\,J^\Sigma\Phi_t(x)\,,\qquad x\in\Sigma\,, t\in(0,R_\Sigma(x))\,,
  \\
  &0\,,\qquad\hspace{3.56cm} x\in\Sigma\,, t\ge R_\Sigma(x)\,.
\end{split}
\right .
\]
By \eqref{prop: hk and n,h sets: eq5}, \eqref{prop: hk and n,h sets: eq6}, and the non-negativity of $f$, $H$ and $J^\Sigma\Phi_t$, we see that $t\in[0,\infty)\mapsto g(x,t)$ is decreasing, and thus provides a non-negative extension of $J^\Sigma\Phi\,(f\circ\Phi)/H$ from $A_\Sigma$ to the whole $\Sigma\times[0,\infty)$ such that
\[
\L(\Sigma)=(n-1)\,\lim_{t\to\infty}\int_\Sigma\,g(x,t)\,d\H^{n-1}_x\,.
\]
Now, since $\{R_\Sigma^*>t\}\subset \Sigma$ and since $t<R_\Sigma^*(x)$ implies $t<R_\Sigma(x)$ we have
\[
\int_\Sigma\,g(x,t)\,d\H^{n-1}_x\ge\int_{\{R_\Sigma^*>t\}}\Big(\frac{f\circ \Phi}{H}\Big)\,J^\Sigma\Phi_t\,d\H^{n-1}_x
=\int_{\Sigma_t^*}\frac{f}{H_{\Sigma_t^*}}\,d\H^{n-1}\,,
\]
where we have used \eqref{def of Hstart}. Now, by the Cauchy--Schwartz inequality,
\begin{eqnarray*}
\int_{\Sigma_t^*}\frac{f}{H_{\Sigma_t^*}}\,d\H^{n-1}\ge \H^{n-1}(\Sigma_t^*)^2\,\Big(\int_{\Sigma_t^*}\frac{H_{\Sigma_t^*}}{f}\,d\H^{n-1}\Big)^{-1}
\end{eqnarray*}
so that, in summary,
\begin{equation}
  \label{lsigma 1}
  \L(\Sigma)\ge(n-1)\,\limsup_{t\to\infty}\H^{n-1}(\Sigma_t^*)^2\,\Big(\int_{\Sigma_t^*}\frac{H_{\Sigma_t^*}}{f}\,d\H^{n-1}\Big)^{-1}\,.
\end{equation}
{\it Claim}: for every $\l\in(0,1)$ there is $t_0=t_0(\l)$ so that, if $t>t_0$, then
\begin{eqnarray}
\label{lsigma 4}
 \inf_{\Sigma_t^*}\langle X,\nu_t\rangle_g  \!\!&\ge&\!\!\l\,h(0)\,,
  \\
  \label{lsigma 5}
  (n-1)\,\H^{n-1}(\Sigma_t^*)\!\!&\ge&\!\!
    \int_{\Sigma_t^*}H_{\Sigma_t^*}\,\frac{\langle X, \nu_t\rangle_g}f\,d\H^{n-1}
  \,,
\end{eqnarray}
with $\nu_t$ as in \eqref{def of nu}. Notice that by combining \eqref{lsigma 1}, \eqref{lsigma 4} and \eqref{lsigma 5} we obtain indeed that
 \begin{eqnarray*}
 \L(\Sigma)&\ge&
 (n-1)\,\limsup_{t\to\infty}\inf_{\Sigma_t^*}\langle X,\nu_t\rangle_g \,\H^{n-1}(\Sigma_t^*)^2\,\Big(\int_{\Sigma_t^*}\frac{H_{\Sigma_t^*}}{f}\,\langle X,\nu_t\rangle_g\,\Big)^{-1}
\\
&\ge& \l\,h(0)\,\limsup_{t\to\infty}\,\H^{n-1}(\Sigma_t^*)\ge\l h(0)^n\,\vol(N)\,,
\end{eqnarray*}
where in the last step we have used \eqref{lsigma 3}. By letting $\l\to 1^-$ we deduce \eqref{claim}. We are thus left to prove \eqref{lsigma 4} and \eqref{lsigma 5} to complete the proof of conclusions (b) and (c).

\medskip

\noindent {\it Proof of \eqref{lsigma 4}}: Recalling that $\Om_t=\{x\in\Om:\dist_{g^*}(x,\ov\Sigma)>t\}$, we now consider
\begin{eqnarray*}
u_{\Om_t}^*=\dist_{g^{*}}(\cdot,\Om_t)\,,\qquad t>0\,,
\end{eqnarray*}
and notice that, with the same argument used in the proof of \cite[Lemma 3.6]{brendle}, for every $\l\in(0,1)$ there is $t_0=t_0(\l)$ such that, if $p\in\Om_{t_0}$ and $\a$ is a $g^{*}$-unit speed geodesic with $\a(0)=p$ and $\a(u^*_\Sigma(p))\in\ov\Sigma$, then $|\a'(0)|_g=f(p)$ and
\begin{equation}
  \label{brendle 36}
  \langle \alpha'(0), \pa/\pa r\rangle_{g} \geq\l\,f(p)\,.
\end{equation}
(In more geometric terms, every $g^{*}$-unit speed geodesic that ends up in $\ov{\Sigma}$ after originating in $\Om$ from a point at a sufficiently large distance from $\ov\Sigma$, must have an initial velocity with ``almost vertical'' direction).  If $t>t_0$ we can apply this statement to any $p\in\Sigma_t^*\subset\Sigma_t\subset\Om_{t_0}$ and with $\a'(0)=f(p)\,\nu_t(p)$, so to find
 \[
\langle\nu_t(p),\pa/\pa r\rangle_g\ge\l \qquad\forall p\in\Sigma_t^*\,,
\]
from which \eqref{lsigma 4} follows since $X=h\,\pa/\pa r$ and $h\ge h(0)$ on $M$.

\medskip

\noindent {\it As an additional consequence of \eqref{brendle 36}}, setting from now on $t_0=t_0(\l)|_{\l=1/2}$, we also notice that, for every $t \geq t_0$,
 \begin{equation}
  \label{has positive reach}
  \mbox{$\Omega_t$ has  positive reach in $(M,g^{*})$}\,.
\end{equation}
To prove this, thanks to Theorem \ref{thm bangert}, we only need to show that $u_\Sigma^*$ has no critical points in $\Om_{t_0}$ ($t_0=t_0(\l)|_{\l=1/2}$), that is, that there cannot be $p\in\Om_{t_0}$ such that for every $v\in T_pM$ with $|v|_{g^*}=1$ one can find a $g^*$-unit speed geodesic $\a$ with $\a(0)=p$ and $\a(u_\Sigma^*(p))\in\ov\Sigma$ such that $\langle v,\a'(0)\rangle_{g}\ge 0$; and, indeed, any such $\a$ would satisfy $\langle\a'(0),\pa/\pa r\rangle_{g^*}\ge f(p)/2$ by \eqref{brendle 36}, so that, taking $v=-f(p)\,(\pa/\pa r)$, we would obtain a contradiction.

\medskip

\noindent {\it Proof of \eqref{lsigma 5}}: The proof is based on an approximation argument. Precisely, for $t>0$ and $\e<\min\{1,t\}$, we consider the sets
\[
W_{t,\e}=\big\{x \in M:u_{\Om_t}^*(x)=\e\big\}\,,
\]
so that
\begin{equation}
  \label{Wteps inclusions}
  \Sigma_{t-\e}^*\subset W_{t,\e}\subset\Om_{t-1}\setminus\ov{\Om}_t\,,\qquad\forall t>0\,,\e<\min\{1,t\}\,,
\end{equation}
and reduce the proof of \eqref{lsigma 5} to showing that
\begin{eqnarray}
\label{pino}
&&\lim_{\e\to 0^+}\H^{n-1}( W_{t,\e})=\H^{n-1}(\Sigma_t)\,,
\\
\label{prop: hk and n,h sets: eq12}
&&\lim_{\e\to 0^+}\int_{\Sigma^\ast_{t-\e}}\!\!\! H_{\Sigma^\ast_{t-\e}} \frac{\langle X, \nu_{t-\e}\rangle_g}f \, d\H^{n-1}= \int_{\Sigma^\ast_t}H_{\Sigma^\ast_t}\, \frac{\langle X, \nu_t\rangle_g}f\, d\H^{n-1}\,,\hspace{1cm}
\\
\label{ice}
&&\liminf_{\e\to 0^+}\int_{ W_{t,\e}\setminus\Sigma^*_{t-\e}}\,\frac{\langle \vec{H}_{ W_{t,\e}},X\rangle_g}f\,d\H^{n-1}
\ge 0\,.
\end{eqnarray}
Indeed, thanks to ${\rm (H1)}$, there is $r_1>0$ such that $ h''>0$ on $(0,2\,r_1)$. Up to further increase the value of $t_0$, we can ensure $\Omega_{t-1}\subseteq N\times[0,r_1]$ for every $t\ge t_0$. In particular, by $W_{t,\e}\subset\Om_{t-1}$, we conclude that $h''>0$ on $ W_{t,\e}$. Since $ W_{t,\e}$ is a closed $C^{1,1}$-hypersurface in $M^\circ$ with $\H^{n-1}( W_{t,\e})<\infty$, by \eqref{lem: integration by parts eq2} we find
\begin{eqnarray}\label{quante}
&&(n-1)\,\H^{n-1}( W_{t,\e})\ge\int_{ W_{t,\e}}\,\frac{\langle \vec{H}_{ W_{t,\e}},X\rangle_g}f \,d\H^{n-1}
\\\nonumber
&&=\int_{ W_{t,\e}\setminus\Sigma^*_{t-\e}}\,\frac{\langle \vec{H}_{ W_{t,\e}},X\rangle_g}f\,d\H^{n-1}
+\int_{\Sigma^*_{t-\e}}\,H_{\Sigma^*_{t-\e}}\,\frac{\langle \nu_{t-\e},X\rangle_g}f\,d\H^{n-1}
\end{eqnarray}
where we have used $\Sigma^*_{t-\e}\subset W_{t,\e}$ to deduce
\[
\vec{H}_{ W_{t,\e}}=\vec{H}_{\Sigma^*_{t-\e}}=H_{\Sigma^*_{t-\e}}\,\nu_{t-\e}\quad\mbox{on $\Sigma^*_{t-\e}$}\,.
\]
By using \eqref{pino}, \eqref{prop: hk and n,h sets: eq12}, \eqref{ice}, and \eqref{quante} we deduce immediately \eqref{lsigma 5}.
We now turn to the proof of \eqref{pino}, \eqref{prop: hk and n,h sets: eq12}, and \eqref{ice}. We shall use the following preliminary step:

\medskip

\noindent {\it Proof of \eqref{pino} and \eqref{prop: hk and n,h sets: eq12}}: Thanks to \eqref{has positive reach} we can apply Theorem \ref{thm kleinjohann} with $A=\ov{\Om}_t$, $t\ge t_0$ and find $\e_t\in(0,\min\{1,t\})$ such that, if $\e\in(0,\e_t)$, then $ W_{t,\e}=\{u_{\Om_t}^*=\e\}$ is a compact $C^{1,1}$-hypersurface contained in $\Unp(\Om_t)$ and the $g^*$-geodesic flow $\Psi^t:\mathcal{N}^1(\ov{\Om}_t)\times[0,\e_t)\to \{0\le u_{\Om_t}^*<\e_t\}$ defined by
\[
\Psi^{t}(p,\eta,\e)=\Psi_\e^t(p,\eta)=\exp^*(p,\e\,\eta)\,,\qquad (p,\eta,\e)\in\mathcal{N}^1(\ov{\Om}_t)\times[0,\e_t)\,,
\]
is such that $\Psi^t_\e$ is a locally bi-Lipschitz map $\mathcal{N}^1(\ov{\Om}_t)$ to $ W_{t,\e}$ when $\e>0$, with $\Psi^t_0(p,\eta)=p$. In particular,
\[
\Theta^t:=\mathcal{N}^1(\ov{\Om}_t)
\]
is a $(n-1)$-dimensional compact Lipschitz submanifold of $TM$. On noticing that  $\mathcal{N}^1_p(\ov{\Omega}_t)=\varnothing$ for every $p\in\Om_t\cup N_0$, since $\pa\Om_t=\Sigma_t\cup N_0$ we find that
\begin{equation}
  \label{utile 2}
  \mathcal{N}^1(\ov{\Omega}_t)=  \mathcal{N}^1(\ov{\Omega}_t)\,\llcorner\,\Sigma_t\,,\qquad\mbox{i.e. $(p,\eta)\in\Theta^t$ implies $p\in\Sigma_t$}\,.
\end{equation}
Moreover, by Theorem \ref{thm normal bundles closed sets}-(iii) (applied to $\ov{\Om}_t$, see, in particular, \eqref{utile}),
\begin{equation}
  \label{utile 3}
  \H^0(\mathcal{N}^1_p(\ov{\Omega}_t))=1\,,\qquad\mbox{for $\H^{n-1}$-a.e. $p\in\Sigma_t$}\,.
\end{equation}
Finally, the smoothness of $(p,\eta)\in TM\mapsto\exp^*(p,\e\,\eta)\in M$ ensures that
\begin{equation}
  \label{utile 4}
  \lim_{\e\to 0^+}J^{\Theta^t}\Psi^t_\e(p,\eta)=J^{\Theta^t}\Psi^t_0(p,\eta)\,,
\end{equation}
for $\H^{n-1}$-a.e. $(p,\eta)\in\Theta^t$ (i.e., at every $(p,\eta)$ such that $T_{(p,\eta)}\Theta^t$ exists). Since $\Psi^t_\e$ and its differential are locally bounded in $(M,g)$, we can apply the area formula (to $\Psi^t_\e$), the dominated convergence theorem (in combination with \eqref{utile 4}), and the area formula again (to $\Psi^t_0$) to find
\begin{eqnarray*}
&&\lim_{\e\to 0^+}\H^{n-1}(W_{t,\e})=\lim_{\e\to 0^+}\int_{\Theta^t}J^{\Theta^t}\Psi^t_\e\,d\H^{n-1}=\int_{\Theta^t}J^{\Theta^t}\Psi^t_0\,d\H^{n-1}
\\
&&=\int_{\Psi^t_0(\Theta^t)}\!\!\H^0\big((\Psi^t_0)^{-1}(x)\big)\,d\H^{n-1}_x
=\int_{\Psi^t_0(\Theta^t)}\!\!\H^0\big(\mathcal{N}_x^1(\ov{\Om}_t)\big)\,d\H^{n-1}_x\,,
\end{eqnarray*}
which, combined with \eqref{utile 3} and $\Psi^t_0(\Theta^t)\subset\Sigma_t$ (i.e. \eqref{utile 2}), gives
\begin{equation}\label{prop: hk and n,h sets: eq10}
	\lim_{\e\to 0^+}\H^{n-1}( W_{t,\e})=
\int_{\Psi^t_0(\Theta^t)}\!\!\H^0\big(\mathcal{N}_x^1(\ov{\Om}_t)\big)\,d\H^{n-1}_x\le\H^{n-1}(\Sigma_t)\,.
\end{equation}
Now, to prove \eqref{pino}, let us consider the diffeomorphisms $\phi_\e : \Sigma^\ast_t \rightarrow \Sigma^\ast_{t-\e}$ defined by
\[
\phi_\e(x)=\Phi_{t-\e}\big(\Phi_t^{-1}(x)\big)=\exp^*(x,\e\,f(x)\,\nu_t(x))\,,\qquad\forall x\in\Sigma_t^*\,.
\]
By the area formula,
\[
\H^{n-1}(\phi_\e(\Sigma^\ast_t)) =\int_{\Sigma^\ast_t} J^{\Sigma^\ast_t}\phi_\e\, d\H^{n-1}\,.
\]
Since $\phi_\e\to{\rm Id}$ and $J^{\Sigma^\ast_t}\phi_{\e}\to1$ on $\Sigma^\ast_t$ as $\e\to 0$, we conclude by dominated convergence that
\begin{equation}\label{prop: hk and n,h sets: eq11}
	\lim_{\e\to 0^+}\H^{n-1}(\Sigma^\ast_{t-\e})= \H^{n-1}(\Sigma^\ast_t).
\end{equation}
Thus, by combining \eqref{prop: hk and n,h sets: eq10} and \eqref{prop: hk and n,h sets: eq11} with the facts that $\Sigma^\ast_{t-\e}\subset W_{t,\e}$ and $\Sigma_t^*$ is $\H^{n-1}$-equivalent to $\Sigma_t$ (recall \eqref{prop: hk and n,h sets: eq13}), we deduce \eqref{pino} and
\begin{equation}
\label{prop: hk and n,h sets: eq16}
\lim_{\e\to 0^+}\H^{n-1}( W_{t,\e}\setminus \Sigma^\ast_{t-\e})=0\,.
\end{equation}
To prove \eqref{prop: hk and n,h sets: eq12} we notice that, since $(\Phi^{-1})|_{\Sigma_s^*}=(\Phi_s)^{-1}$ on $\Sigma_s^*$, we have
\[
(\Phi^{-1})|_{\Sigma^*_{t-\e}}\circ\phi_\e=(\Phi_t)^{-1}=(\Phi^{-1})|_{\Sigma^*_t}\qquad\mbox{on $\Sigma_t^*$}\,,
\]
so that,  the area formula gives
\begin{eqnarray*}
\int_{\Sigma^\ast_{t-\e}}\!\!\! H_{\Sigma_{t-\e}^*} \frac{\langle X, \nu_{t-\e}\rangle_g}{f} \, d\H^{n-1}\!\!\!&=&\!\!\!
\int_{\Sigma^\ast_{t-\e}}\!\!\! (H\circ (\Phi_{t-\e})^{-1})\, \frac{\langle X, \nu_\Om\circ(\Phi_{t-\e})^{-1}\rangle_g}f \,
\\
\!\!\!&=&\!\!\!\int_{\Sigma^\ast_{t}} H_{\Sigma^*_t}\, \frac{\langle X\circ\phi_\e, \nu_t\rangle_g}{f\circ\phi_\e} \,J\phi_\e\, d\H^{n-1}\,,
\end{eqnarray*}
for every $\e>0$. Then \eqref{prop: hk and n,h sets: eq12} follows by dominated convergence.

\medskip

\noindent {\it Proof of \eqref{ice}}: We finally prove \eqref{ice}, that is,
\begin{equation}
  \label{ice2}
  \liminf_{\e\to 0^+}\int_{ W_{t,\e}\setminus\Sigma^*_{t-\e}}\,\frac{\langle \vec{H}_{ W_{t,\e}},X\rangle_g}f\,d\H^{n-1}
\ge 0\,.
\end{equation}
Setting, for the sake of brevity,
\begin{eqnarray}\label{def of nut eps}
\nu_{t,\e}=\frac{\nabla u^*_{\Om_t}}{|\nabla u^*_{\Om_t}|_{g}}\,,\qquad H_{t,\e}=\vec{H}_{ W_{t,\e}}\cdot\nu_{t,\e}\,,\qquad\mbox{on $W_{t,\e}$}\,,
\end{eqnarray}
we notice that $\nu_{t,\e}$ defines a Lipschitz continuous $g$-unit normal to $ W_{t,\e}$, and that $H_{t,\e}$ is the scalar mean curvature (as usual, with respect to $g$) of $ W_{t,\e}$ relative to $\nu_t$. With this notation,  and thanks to \eqref{prop: hk and n,h sets: eq16}, \eqref{ice2} follows by showing that
\begin{eqnarray}
  \label{form of ice 1}
  &&\inf_{ W_{t,\e}}\langle \nu_{t,\e},\pa/\pa r\rangle_g\ge0\,,
  \\
  \label{form of ice 2}
  &&H_{t,\e}(x)\ge-\Lambda(t)\,,\qquad\mbox{for $\H^{n-1}$-a.e. $x\in W_{t,\e}$}\,,
\end{eqnarray}
for a (positive) constant $\Lambda(t)$ independent of $\e$.

\medskip

\noindent {\it Proof of \eqref{form of ice 1}}: We start by proving the existence of a positive constant $c$ such that
\begin{equation}\label{prop: hk and n,h sets: eq14}
\big\langle \eta, X(p)\big \rangle_{g^{*}} \geq c\,,\quad\forall (p,\eta)\in\Theta_t\,,t>t_0\,.
\end{equation}
Since $\langle\eta,v\rangle_{g^{*}} \leq 0$ whenever $\eta \in \mathcal{N}_p(\ov{\Omega}_t)$ and $v \in T_p(\Omega_t)$ (where $T_p(\Om_t)$ is tangent cone to $\Om_t$ at $p$), and recalling \eqref{utile 3}, the validity of \eqref{prop: hk and n,h sets: eq14} (for an explicitly computable constant $\s_0$) can be easily deduced by showing that, for every $t>t_0$ and $p \in \Sigma_t\subset\Om_{t_0}$,
\begin{equation}
  \label{imp}
  \left\{\begin{split}
  &v\in T_pM\,,\quad |v|_{g} = 1\,,
  \\
  &\big\langle v, -(\pa/\pa r)|_p\big\rangle_{g}>\frac{15}{16}\,,
  \end{split}
  \right .
  \qquad \Rightarrow\qquad v\in T_p(\Om_t)\,.
\end{equation}
(In geometric terms: leaving $p\in\Sigma_t=(\pa\Om_t)\cap M^\circ$ along a sufficiently ``vertical and downward'' direction, we stay inside $\Om_t$.) The proof of \eqref{imp} follows closely that of \cite[Lemma 3.7]{brendle}, but since the two statements are not immediate to compare, we include the details. We need to consider an arbitrary $g$-unit speed curve $\gamma:[0,1]\to M$ with $\g(0)=p$ and $\g'(0)=v$, and prove the existence of $\s\in(0,1)$ such that $u^*_\Sigma(\g(s))>t$ for every $s\in(0,\s)$. To begin with, we can definitely chose $\s$ so that
\[
\g(s)\in\Om_{t_0}\,,\qquad \big\langle \gamma'(s), -(\pa/\pa r)|_{\gamma(s)}\big\rangle_{g} \geq \frac{15}{16}\,,\quad\forall s\in[0,\s]\,.
\]
For every $s\in(0,\s)$, the fact that $\g(s)\in\Om_{t_0}$ implies $u^*_\Sigma(\gamma(s))>0$, and thus the existence of a $g^*$-unit speed geodesic $\a_s$ with $\a_s(0)=\gamma(s)$ and $\a_s(u^*_\Sigma(\gamma(s)))\in\overline{\Sigma}$. By \eqref{brendle 36} (with $\l=1/2$), since $w(s)=\alpha_s'(0)/f(\g(s))$ is a $g$-unit vector, we find
\begin{eqnarray*}
\frac{1}2\!\!&\le&\!\!\big\langle w(s), (\pa/\pa r)|_{\gamma(s)}\big\rangle_{g}
=
\big\langle w(s),\g'(s)+ (\pa/\pa r)|_{\gamma(s)}\big\rangle_{g}-\langle w(s), \g'(s)\rangle_{g}
\\
\!\!&\le&\!\!\big|\g'(s)+(\pa/\pa r)|_{\gamma(s)}\big|_{g}-\langle w(s), \g'(s)\rangle_{g}
\\
\!\!&=&\!\!\sqrt{2-2\,\langle \gamma'(s),-(\pa/\pa r)|_{\gamma(s)}\rangle_{g}}-\langle w(s), \g'(s)\rangle_{g}
\\
\!\!&\le&\!\! \frac14-\langle w(s), \g'(s)\rangle_{g}\,,\qquad\mbox{i.e.}\quad \langle -w(s), \g'(s)\rangle_{g}\ge \frac14\quad\forall s\in(0,\s)\,.
\end{eqnarray*}
Using the facts that $u^*_\Sigma(p)=t$ and that, for a.e. $s\in(0,\s)$, $u_\Sigma^*\circ\g$ is differentiable at $\g(s)$ with
\[
(u_\Sigma^*\circ\g)'(s)=\langle\nabla^*u_\Sigma^*(\g(s)),\g'(s)\rangle_{g^*}=\langle-\a'_s(0),\g'(s)\rangle_{g^*}
=\frac{\langle\g'(s), -w(s)\rangle_{g}}{f(\g(s))}\,,
\]
we thus find that, for every $s'\in(0,\s)$,
\[
u^*_\Sigma(\gamma(s'))
=t+ \int_{0}^{s'} \frac{ds}{4\,f(\g(s))} >t
\]
as desired. This proves \eqref{imp}, and thus, as explained, \eqref{prop: hk and n,h sets: eq14}.

\medskip

We are now ready to deduce \eqref{form of ice 1} from \eqref{prop: hk and n,h sets: eq14}. First, by \eqref{prop: hk and n,h sets: eq14} there are positive constants $c$ and $\de$ such that
\begin{equation}
  \label{prop: hk and n,h sets: eq14 bis}
  \langle\eta',\pa/\pa r\rangle_g\ge c\,,
\end{equation}
whenever $(q,\eta')$ lies in the $\de$-neighborhood $\mathcal{A}_\de(\Theta^t)$ of $\Theta^t=\mathcal{N}^1(\ov{\Om}_t)$ in $TM$. Second, by smoothness of  $(p,\eta,\e)\mapsto\exp^*(p,\e\,\eta)$, we can find $\e'<\e_t$ depending on $\de$ such that
\[
\Big\{\Big(\Psi^t(p,\eta,\e),\frac{\pa\Psi^t}{\pa\e}(p,\eta,\e)\Big):(p,\eta)\in\Theta^t\,,0<\e<\e'\Big\}\subset\mathcal{A}_\de(\Theta^t)\,.
\]
Third, since $\Psi^t_\e$ is a bijection from $\Theta^t$ to $W_{t,\e}$, we see that for each $x\in W_{t,\e}$ there is a unique pair $(p,\eta)$ with $p\in\Sigma_t$ and  $\eta\in\mathcal{N}^1_p(\ov{\Om}_t)$ such that $x=\Psi^t_\e(p,\eta)$. Thus, taking also into account that (in general) $|\nabla^* v|_{g^*}=f\,|\nabla v|_g$ (by $\nabla^*v=f^2\,\nabla v$), that $|\nabla^*u^*_{\Om_t}|_{g^*}=1$ (wherever $u^*_{\Om_t}$ is differentiable), and that $u_{\Om_t}^*$ is differentiable along $\e\mapsto\Psi^t_\e(p,\eta)$ with $g^*$-gradient given by $\pa\Psi^t_\e/\pa\e$, we conclude that
\begin{eqnarray*}
\langle\nu_{t,\e}(x),\pa/\pa r\rangle_g&=&
\Big\langle\frac{\nabla u_{\Om_t}^*(\Psi^t_\e(p,\eta))}{|\nabla u_{\Om_t}^*(\Psi^t_\e(p,\eta))|_g},\pa/\pa r\Big\rangle_g
\\
&=&f\big(\Psi^t_\e(p,\eta)\big)\,\langle \nabla u_{\Om_t}^*(\Psi^t_\e(p,\eta)),\pa/\pa r\rangle_g
\\
&=&f\big(\Psi^t_\e(p,\eta)\big)\,\big\langle \frac{\pa\Psi^t}{\pa\e}(p,\eta,\e),\pa/\pa r\big\rangle_g\ge c\,\inf_{[a,b]}h'>0\,,
\end{eqnarray*}
provided $\e<\e'$. This proves  \eqref{form of ice 1}.

\medskip

\noindent {\it Proof of \eqref{form of ice 2}}: Let us consider the open set
\[
A_{t,\e}=\Om_t\cup\big\{p\in\Om:u_{\Om_t}^\ast(p)<\e\big\}\,,
\]
that has $C^{1,1}$-boundary and $g$-unit outer unit normal given by $\nu_{A_{t,\e}}=\nu_{t,\e}$ on $\pa A_{t,\e}=W_{t,\e}$, with $\nu_{t,\e}$ as in \eqref{def of nut eps}. Given $x\in W_{t,\e}$ such that $\nu_{t,\e}$ is differentiable at $x$ (this holds at $\H^{n-1}$-a.e. $x\in W_{t,\e}$), let  $B^*_x$ denote the $g^*$-geodesic ball centered at $\exp^*(x,(\mu_t/2)\,\nabla^*u_{\Om_t}^*(x))$, where $\mu_t<\min\{1,\e_t\}$ is smaller than the injectivity radius of $M$ in $\ov{\Om}_{t-2}\setminus\Om_{t+1}$, so to entail that $B^*_x$ has smooth boundary. Since, by construction,
\[
x\in\pa A_{t,\e}\cap\pa B^*_x\,,\qquad A_{t,\e}\subset M\setminus B^*_x\,,
\]
by the weak maximum principle for $C^{1,1}$-vs-smooth hypersurfaces we find
\[
H_{W_{t,\e}}(x)=\vec{H}_{A_{t,\e}}(x)\cdot\nu_{t,\e}(x)\ge \vec{H}_{M\setminus B^*_x}\cdot\nu_{t,\e}(x)=-H_{B^*_x}(x)\,,
\]
where $H_{B^*_x}(x)$ denotes the scalar mean curvature of $B^*_x$ in $(M,g)$, computed with respect to the outer $g$-unit normal to $B^*_x$ at $x$  (i.e., with respect to $-\nu_{t,\e}(x)$). Now, if $H^*_{B^*_x}$ denotes the scalar mean curvature of $B^*_x$ in $(M,g^*)$ computed with respect to the outer $g^*$-unit normal to $B^*_x$ at $x$, then by \eqref{lem: mh sets and conformal metric: eq1} we find
\begin{eqnarray*}
  H_{B^*_x}(x)\!\!&\le&\!\!H^*_{B^*_x}(x)+3\,(n-1)\,|\nabla\log(f)(x)|_g
  \\
  \!\!&\le&\!\!(n-1)\,\sqrt{\k_t}\,{\rm coth}(\sqrt{\k_t}\,\mu_t/2)+3\,(n-1)\,|\nabla\log(f)(x)|_g\,,
\end{eqnarray*}
where in the last inequality we have denoted by $-\k_t$ a negative lower bound for the sectional curvatures of $(M,g^*)$ in $\ov{\Om}_{t-2}\setminus\Om_{t+1}$, and have used \cite[pag.\ 184]{Karchercomparison} (comparison with the mean curvature of geodesic balls in an hyperbolic model space). Since the right-hand side can be bounded by a positive constant $\Lambda(t)$, we have concluded the proof of \eqref{form of ice 2}, and thus of conclusions (b) and (c) of the theorem.
\end{proof}

\begin{proof}[Proof of Theorem \ref{thm rigidity}] {\it Preparation}: As in the proof of Theorem \ref{thm hk inq}, we reduce the case when $M=N\times(0,\infty)$,  ${\rm (H0)}$ and ${\rm (H1)}$ hold, ${\rm (H2)}$ holds on $(0,\infty)$ (i.e., $h'>0$ on $(0,\infty)$), ${\rm (H3)}$ and ${\rm (H4)}$ hold on $(0,b)$ (if $(M,g)\in\B_n$) {\bf or} ${\rm (H3)^*}$ holds on $(0,b)$ (if $(M,g)\in\B_n^*$); $f=h'\circ r$ is bounded on $M$ and the metric $g^*=f^{-2}\,g$ is such that $(M,g^*)$ is a complete Riemannian manifold; and, finally, $\Sigma$ and $\Omega$ satisfy \eqref{ab} in addition to assumptions ${\rm (A1)}$, ${\rm (A2)}$', and ${\rm (A3)}$'. Next, by testing \eqref{vale div thm} with vector fields compactly supported in $M\setminus(\ov\Sigma\setminus\Sigma)$ we see that
\[
H_\Sigma\equiv H_0\qquad\mbox{on $\Sigma$}\,,
\]
while testing \eqref{vale div thm} with $X=h\,\pa/\pa r$ and taking into account that $\Div^\Sigma X=(n-1)\,f$ on $\Sigma$ and $\Div(X)=n\,f$ on $M$ by \eqref{Div Gamma X}, we find, in the case $\pa\Om=\ov\Sigma$,
\begin{eqnarray*}
  (n-1)\,\int_\Sigma f\,d\H^{n-1}&=&\int_\Sigma\Div^\Sigma X\,d\H^{n-1}=H_0\,\int_\Sigma \langle X,\nu_\Om\rangle_g\,d\H^{n-1} \\
  &=&H_0\,\int_\Om\Div X=H_0\,n\,\int_\Om f\,,
\end{eqnarray*}
i.e., $H_\Sigma\equiv H_0>0$ (so that ${\rm (A2)}$ holds) and \eqref{conclusion a} holds as an equality; and, in the case $\pa\Om=\ov\Sigma\cup N_0$,
\begin{eqnarray*}
  (n-1)\,\int_\Sigma f\,d\H^{n-1}&=&\int_\Sigma\Div^\Sigma X\,d\H^{n-1}=H_0\,\int_\Sigma \langle X,\nu_\Om\rangle_g\,d\H^{n-1} \\
  &=&H_0\,\int_\Om\Div X\,d\H^n-\int_{N_0}\langle X,\pa/\pa r\rangle_g\,d\H^{n-1}
  \\
  &=&H_0\,n\,\int_\Om f-h(0)^n\,\vol(N)\,,
\end{eqnarray*}
i.e, $H_\Sigma\equiv H_0 >0$ (so that ${\rm (A2)}$ holds) and \eqref{conclusion b} holds as an equality; finally, Theorem \ref{thm white} and \eqref{vale div thm} imply the validity of assumption ${\rm (A3)}$. We can thus apply conclusion (c) of Theorem \ref{thm hk inq} to conclude that
\begin{eqnarray}
\label{dai 1}
  &&\mbox{$\Sigma$ is umbilical and has constant mean curvature in $(M,g)$}\,,
  \\\label{dai 2}
  &&\mbox{$M\setminus\Om$ has positive reach in $(M,g^*)$}\,,
  \\
  \label{dai T}
  &&\mbox{$(\mathcal{M}[h])'\,g_N(\nu_\Om,\nu_\Om)=0$ on $\Sigma$}\,,
\end{eqnarray}
with $\mathcal{M}[h]$ defined as in \eqref{def of M h}.

\medskip

\noindent {\it Conclusion of the proof if $(M,g)\in\B_n^*$}: In this case, \eqref{ab} and the validity of ${\rm (H3)^*}$ on $(0,b)$ implies that $(\mathcal{M}[h])'>0$ on $\Sigma$, so that \eqref{dai T} gives
\begin{equation}
  \label{normals facile}
  \mbox{if $p\in\Sigma$, then $\nu_\Om(p)$ is parallel to $(\pa/\pa r)|_p$}\,.
\end{equation}
Now, if $\phi$ is a local chart of $N$, defined on a ball $B$ in $\R^{n-1}$, then $\psi(x,t)=(\phi(x),t)$ defines a local chart of $M$ defined on the open set $V=B\times(0,\bar{r})$; clearly, $\psi^{-1}(\Om)$ is a set of finite perimeter in $V\subset\R^n$, with $\nu_{\psi^{-1}(\Om)}$ parallel to $e_n$ $\H^{n-1}$-a.e. on $\pa^*[\psi^{-1}(\Om)]$; by \cite[Exercise 15.18]{maggiBOOK}, $\psi^{-1}(\Om)\cap V$ is $\H^n$-equivalent to $B\times J$, where $J$ is a finite union of open intervals compactly contained in $(0,\bar{r})$; covering $N$ by such charts $\phi$, and going back to $M$, we conclude that $\Om$ is $\H^n$-equivalent to a $N_0\times J$. Then, the fact that $H_\Sigma$ is constant implies that $J$ is either equal to $(0,t_0)$ or to $(t_0,\bar{r})$ for some $t_0\in(0,\bar{r})$, and the theorem is proved.

\medskip

\medskip

\noindent {\it Conclusion of the proof if $(M,g)\in\B_n$}: Condition \eqref{dai 1} combined with the Codazzi equations implies that
\begin{equation}
\label{dai 3old}
  ({\rm Ric}_M)_p(\nu_\Om(p),\s_i(p))=0\,,\qquad\forall p\in\Sigma\,,i=1,...,n-1\,,
\end{equation}
provided $\{\s_i(p)\}_{i=1}^{n-1}$ is an $g$-orthonormal basis of $T_p\Sigma$; in particular,
\begin{equation}
    \label{dai 3}
    \mbox{if $p\in\Sigma$, then $\nu_\Om(p)$ is an eigenvector of ${\rm Ric}_p$}\,.
\end{equation}
Since ${\rm (H4)}$ implies that $(\pa/\pa r)|_p$ is a {\it simple} eigenvector of $({\rm Ric}_M)$ (with  eigenvalue $-(n-1)\,(h''/h)(r(p))$), it follows from \eqref{dai 3} that
\begin{equation}
  \label{normals}
  \mbox{if $p\in\Sigma$, then $\nu_\Om(p)$ is either orthogonal or parallel to $(\pa/\pa r)|_p$}\,.
\end{equation}
Now, by \eqref{ab}, there is $t_0>0$ and $p_0\in\ov\Sigma$ such that
\begin{equation}
  \label{def of r0}
  N\times(0,t_0)\subset\Om\,,\qquad p_0\in\ov\Sigma\cap N_{t_0}\subset\pa\Om\,.
\end{equation}
From here, {\it in the smooth case when $\ov\Sigma=\Sigma$} (i.e., in the case considered in \cite{brendle}), \eqref{normals} and \eqref{def of r0} immediately imply, first, that $\nu_\Om(p_0)=(\pa/\pa r)|_{p_0}$, and, second, that $N_{t_0}\subset \Sigma$; from which a sliding argument (also required and detailed below in the non-smooth case) proves the theorem.

\medskip

However, {\it in the non-smooth case}, we cannot immediately conclude the containment of $N_{t_0}$ into $\Sigma$, and, actually, it is not even clear if $\ov\Sigma$ is regular at the contact point $p_0$ defined in \eqref{def of r0}: indeed, the blow-up of (the multiplicity one varifold associated to) $\ov\Sigma$ at $p_0$ is an hyperplane with multiplicity possibly higher than one -- thus preventing the use of Allard's regularity theorem to infer $p_0\in\Sigma$. To exit this impasse we make crucial use of the positive reach property \eqref{dai 2}, which we use to prove the following {\it approximation property}:  for every $(p,\eta)\in\mathcal{N}^1(M\setminus\Om)$ there are a connected component $\Sigma'$ of $\Sigma$ and a sequence $\{p_j\}_j\subset\Sigma'$ such that
\begin{equation}
  \label{closure prop}
  \big(p_j,-f(p_j)\,\nu_{\Om}(p_j)\big)\to(p,\eta)\quad\mbox{in $\mathcal{N}^1(M\setminus\Om)$}\,.
\end{equation}
Indeed, by \eqref{dai 2} and Theorem \ref{thm kleinjohann}, there is $s_0>0$ such that, for every $s\in(0,s_0)$, the sets
\[
Z_s=\big\{x\in M:\dist^*(x,M\setminus\Om)=s\big\}\,,
\]
are $C^{1,1}$-hypersurfaces, and the map $\Psi_s(p,\eta)=\exp^*(p,s\,\eta)$ is bi-Lipschitz from $\mathcal{N}^1(M\setminus\Om)$ to $Z_s$. We notice that
\[
\mathcal{N}^1(M\setminus\Om)\lfloor\Sigma=\big\{(p,-f(p)\,\nu_\Om(p)):p\in\Sigma\big\}\,,
\]
and set
\[
Z_s^*=\Phi_s\big(\mathcal{N}^1(M\setminus\Om)\lfloor\Sigma\big)\subset Z_s\,.
\]
By arguing as in the proof of \eqref{prop: hk and n,h sets: eq13}, with the aid of Theorem \ref{thm Santilli riem} we find that,  for $\L^1$-a.e. $s\in(0,s_0)$, $Z_s^*$ is $\H^{n-1}$-equivalent to $Z_s$. For any such $s$, $Z^\ast_s $ is an open dense subset of $Z_s$, and we can find a sequence $\{q_j\}_j$, contained in same connected component $Z_s^{**}$ of $Z_s^*$, such that $q_j\to\Phi_s(p,\eta)\in Z_s$. Setting $\pi(q,\tau)=q$ for every $(q,\tau)\in TM$, we see that $p_j=\pi[(\Phi_s)^{-1}(q_j)]$ defines a sequence contained in a same connected subset $\pi[(\Phi_s)^{-1}(Z_s^{**})]$ of $\Sigma$, and such that
\[
\big(p_j,-f(p_j)\,\nu_\Om(p_j)\big)=(\Phi_s)^{-1}(q_j)\to (p,\eta)\,,
\]
in $TM$, thus proving \eqref{closure prop}. We now combine \eqref{closure prop} with the fact that, by definition of $p_0$ (recall \eqref{def of r0}) it holds
\begin{equation}
  \label{eta 3}
  -f(p_0)\,\frac{\pa}{\pa r}\Big|_{p_0}\in\mathcal{N}^1_{p_0}(M\setminus\Om)\,,
\end{equation}
to find a sequence $\{p_j\}_j$, contained in a connected component $\Sigma'$ of $\Sigma$, such that $(p_j,-f(p_j)\,\nu_\Om(p_j))\to (p_0,-f(p_0)(\pa/\pa r)|_{p_0})$ as $j\to\infty$. By \eqref{normals}, up to extracting subsequences, there are two alternatives: either
\begin{equation}
  \label{eta 4}
  \mbox{$-f(p_j)\,\nu_\Om(p_j)$ is parallel to $(\pa/\pa r)|_{p_j}$ for every $j$}\,,
\end{equation}
or $g_{p_j}(\nu_\Om(p_j),(\pa/\pa r)|_{p_j})=0$ for every $j$, where the latter is clearly contradictory, since $|\pa/\pa r|_g=1$. By smoothness and connectedness of $\Sigma'$, by \eqref{normals}, and since $\Sigma'$ contains points $p_j$ as in \eqref{eta 4} with $p_j\to p_0$ as $j\to\infty$, we conclude that $\Sigma'$ is an open connected subset of $N_{t_0}$. In fact, it must be $\Sigma'=N_{t_0}$, because the above argument, with $p_0$ replaced by a possible point $p_0'$ in the boundary of $\Sigma'$ relative to $N_{t_0}$, would lead to the contradiction that an open neighborhood of $p_0'$ in $N_{t_0}$ would be contained in $\Sigma'$ itself. We have thus proved that
\begin{equation}
  \label{eta 5}
  N_{t_0}\subset\Sigma\,.
\end{equation}
The same argument also shows that
\begin{equation}
  \label{eta 6}
  N_{t_0}\cap \ov{\Sigma\cap[N\times(t_0,\infty)]}=\varnothing\,.
\end{equation}
By \eqref{eta 6} we could then start sliding $N_t$ upwards to prove that either $\Om=N\times (0,t_0)$ with $M^\circ\cap\pa\Om=\ov\Sigma=\Sigma=N_{t_0}$, thus concluding the proof of the theorem, or we could find $t_1>t_0$ such that
\begin{equation}
  \label{def of t1}
  (t_0,t_1)\setminus M\setminus\ov\Om\,,\qquad p_1\in N_{t_1}\cap\ov\Sigma\subset M\setminus\Om\,.
\end{equation}
By construction, $-f(p_1)\,(\pa/\pa r)|_{p_1}\in\mathcal{N}_{p_1}^1(M\setminus\Om)$, and, by arguing as in the proof of \eqref{eta 5}, we would find
$N_{t_1}\subset\Sigma$, with $\nu_\Om=-(\pa/\pa r)|_{p_1}$ along $N_{t_1}$. In turn, this would give that $H_\Sigma$ is negative along $N_{t_1}$, a contradiction. This finally proves the theorem.
\end{proof}

\section{Rigidity and compactness theorem}\label{section proofs}
\begin{proof}[Proof of Theorem \ref{thm main}] Up to change $\Om$ with $M\setminus\Om$, since $\nu_\Om=-\nu_{M\setminus\Om}$ on $\pa^*\Om=\pa^*(M\setminus\Om)$, we can assume that $\l\ge0$. Let $ V_\Omega $ be the multiplicity one rectifiable varifold associated to $ M^\circ \cap  \partial^\ast \Omega $.  The distributional constant mean curvature condition \eqref{distributional CMC} imply lower density bounds on $\| V_\Omega \|$, which in turn imply
	$$ \mathcal{H}^{n-1}(\overline{\pa^\ast \Omega} \setminus \pa^\ast \Omega) =0. $$
	Therefore, it is not restrictive to assume that $ \Omega $ is an open set such that $ M^\circ \cap \partial \Omega $ is compact, $\overline{\pa^\ast \Omega} = \pa \Omega $ and
	\begin{equation}\label{gino2}
	\mathcal{H}^{n-1}(\pa \Omega \setminus \pa^\ast \Omega) =0.
	\end{equation}
	 (see for instance the construction in the proof of \cite[Lemma 6.2]{deRosaKolaSantilli}). Notice that $ \partial^\ast \Omega = \partial^\ast \overline{\Om} \subseteq \partial \overline{\Om} \subseteq \partial \Om $,
	 hence taking the closure we find that $ \partial \Omega = \partial \overline{\Om} $. By Allard's regularity theorem \cite{Allard}, if we set
\[
\Sigma =\Big\{x\in\spt \| V_\Om \|:\lim_{\rho\to 0^+}\frac{\|V_\Om\|(B_\rho(x))}{\om_{n-1}\,\rho^{n-1}}=1\Big\}\,,
\]
then $\Sigma $ is a smooth, embedded hypersurface and
\begin{equation}
  \label{gino 3}
  \Sigma=M^\circ\cap\pa^*\Om\,.
\end{equation}
We now check that the pair $(\Sigma,\Om)$ satisfies the assumptions ${\rm (A1)}$, ${\rm (A2)}$' and ${\rm (A3)}$' of Theorem \ref{thm rigidity}, thus concluding the proof of the theorem. Clearly, ${\rm (A3)}$' is equivalent to \eqref{distributional CMC}. Since $ M^\circ \cap \partial \Omega $ is compact, we infer that $ \overline{\Sigma} \subseteq M^\circ $; moreover \eqref{gino2} means that $ \mathcal{H}^{n-1}(\overline{\Sigma}\setminus \Sigma) =0 $. Henceforth ${\rm (A1)}$ holds. Concerning ${\rm (A2)}$', since $ M^\circ \cap \partial \Omega $ is compact, we notice that $1_\Om$ is constant in a neighborhood $A$ of $N_0$ in $M$; if $1_\Om=0$ on $A$ then $\pa\Om=\ov\Sigma$; if, otherwise $1_\Om=1$ on $A$, then $N_0\subset\pa\Om$, and thus $\pa\Om=N_0\cup\ov\Sigma$.
\end{proof}

\begin{proof}
  [Proof of Theorem \ref{thm compactness 1}] From $\H^n(\Om_j\Delta\Om)\to 0$ as $j\to\infty$ we easily deduce that for every $x_0\in M^\circ \cap \ov{\pa^*\Om}$ there is $x_j\in\pa^*\Om_j$ such that $x_j\to x_0$ in $M$; for, otherwise, there would be $\rho>0$, with $\ov{B}_\rho(x_0)\cap\ov{\pa^*\Om_j}=\varnothing$ for every $j$, and $X\in C^\infty_c(B_\rho(x_0))$ such that
  \begin{eqnarray*}
  1\!\!&=&\!\!\int_{B_\rho(x_0)\cap \pa^*\Om}\!\!\langle X,\nu_\Om\rangle_g\,d\H^{n-1}=\int_\Om\,\Div\,X\,d\H^n
  \\
  \!\!&=&\!\!\lim_{j\to\infty}\int_{\Om_j}\,\Div\,X\,d\H^n
  =\lim_{j\to\infty}\int_{B_\rho(x_0)\cap\pa^*\Om_j}\!\!\langle X,\nu_{\Om_j}\rangle_g\,d\H^{n-1}=0\,.
  \end{eqnarray*}
  We thus conclude that
  \begin{equation}
    \label{letsapplyit 1}
    M^\circ\cap\ov{\pa^*\Om}\subset N\times[a,b]\cc M^\circ\,.
  \end{equation}
  By $\H^n(\Om_j\Delta\Om)\to 0$ and, crucially, by ${\rm Per}(\Om_j)\to{\rm Per}(\Om)$, as $j\to\infty$, we see that the multiplicity one rectifiable varifolds $V_j$ associated to $\pa^*\Om_j$ converge, in the sense of varifolds on $M$, to the multiplicity one rectifiable varifold $V$ associated to $\pa^*\Om$: in particular, for every $X\in\X(M)$,
  \[
  \lim_{j\to\infty}\int_{M^\circ\cap\pa^*\Om_j}\Div^{\pa^*\Om_j}X\,d\H^{n-1}=\int_{M^\circ\cap\pa^*\Om}\Div^{\pa^*\Om}X\,d\H^{n-1}\,.
  \]
  Again by $\H^n(\Om_j\Delta\Om)\to 0$ as $j\to\infty$ and thanks to the divergence theorem $\small{\int}_{M^\circ\cap\pa^*\Om_j}\langle X,\nu_{\Om_j}\rangle_g\,d\H^{n-1}\to \small{\int}_{M^\circ\cap\pa^*\Om}\langle X,\nu_{\Om}\rangle_g\,d\H^{n-1}$ as $j\to\infty$. Therefore, \eqref{very weak} implies
  \begin{equation}
  \label{letsapplyit 2}
  \int_{M^\circ\cap\pa^*\Om}\Div^{\pa^*\Om}X\,d\H^{n-1}=\l\,\int_{M^\circ\cap\pa^*\Om}\langle X,\nu_{\Om}\rangle_g\,d\H^{n-1}\,.
  \end{equation}
  By \eqref{letsapplyit 1} and \eqref{letsapplyit 2} we can apply Theorem \ref{thm main} to $\Om$ and conclude the proof of the theorem.
\end{proof}

\section{Proof of Theorem \ref{thm main hyperbolic}}

The key observation to prove Theorem \ref{thm main hyperbolic} is contained in the following result, that can be proved employing the same method of Theorem \ref{thm hk inq}.  
\begin{theorem}\label{thm hk in substatic spaces}
	Suppose $(M,g)$ is a $ n $-dimensional Riemannian manifold (notice carefully that we do not assume this space to be geodesically complete) and $ f $ is a smooth positive function on $ M $ such that 
	\begin{equation}\label{eq substatic}
		f {\rm Ric} - \Der^2 f + (\Delta f)g \geq 0 \quad \textrm{on $ M $.}
	\end{equation} 
	Suppose $ \Omega \subseteq M $ is an open set with finite perimeter, with exterior unit-normal $ \nu_\Omega $, such that $ \overline{\Omega} $ is compact, and suppose $ \Sigma \subseteq \partial \overline{\Om} $ is a smooth embedded hypersurface such that $\overline{\Sigma} = \partial \overline{\Om} $ is a White $(n-1, \lambda)$-set of $(M,g)$, $ \mathcal{H}^{n-1}(\overline{\Sigma} \setminus \Sigma) =0 $, and $ H_\Sigma = \overrightarrow{H}_\Sigma \cdot \nu_\Omega $ is positive on $ \Sigma $. 
	
	Then 
	\begin{equation}\label{thm hk in substatic spaces inequality}
		n\int_{\Omega} f\, d\mathcal{H}^n \leq (n-1) \int_\Sigma \frac{f}{H_\Sigma}\, d\mathcal{H}^{n-1}.
	\end{equation}  
	If the equality holds, then there exists $ t_0 > 0 $ such that the sets 
	$$ \Sigma_t = \{p \in  \Omega: \dist_{g^\ast}(p, \overline{\Sigma}) = t\}, \quad \textrm{for $ 0 < t < t_0 $,} $$
	where $ g^\ast = \frac{g}{f^2} $, are closed embedded $ \mathcal{C}^{1,1} $-hypersurfaces, and for $ \mathcal{L}^1 $ a.e.\ $ t \in (0, t_0) $ there exists a smooth embedded umbilical hypersurface $ \Sigma^\ast_t \subseteq \Sigma_t $ such that $$ \mathcal{H}^{n-1}(\Sigma_t \setminus \Sigma^\ast_t) =0. $$
\end{theorem}

\begin{proof}
	Let $ N \subseteq M $ be a compact set  with smooth boundary such that $ \overline{\Omega} \subseteq {\rm int}(N) $. By \cite[Corollary B]{PigolaVeronelliExtension}, there exists a geodesically complete Riemannian extension $ (M^\ast, g^\ast) $ of $ (N, g/f^2) $ with $ \partial M^\ast = \varnothing $. We denote by $ \exp^\ast $ the exponential map of $(M^\ast, g^\ast ) $ and define 
	$$ \Phi: \Sigma \times [0, +\infty) \rightarrow M^\ast, \quad \Phi(x,t) = \Phi_t(x) = \exp^\ast(x, -tf(x)\nu_\Omega(x)) $$
	and, for $ x \in \Sigma $, 
	$$ R_\Sigma(x) = \min\{\inf\{t > 0 : J^\Sigma \Phi_t(x) =0\}, \, \inf\{t > 0 : \Phi_t(x) \notin {\rm int}(N)\}\}. $$
	The conclusion now can be obtained by tracing the argument of Theorem \ref{thm hk inq} that gives the Heintze-Karcher inequality \eqref{conclusion a}. We omit to repeat these details here, and we point out a couple of remarks. First, one needs to employ \eqref{eq substatic} in order to obtain \eqref{screaming trees} and \eqref{prop: hk and n,h sets: eq5} in the present setting: in fact, one can repeat verbatim the pointwise computations of \cite[Proposition 3.2]{brendle}, where only \eqref{eq substatic} is used. Second, to analyze the equality case, firstly we observe, exactly in the same way as in the proof of Theorem \ref{thm hk inq}, that $ M^\ast \setminus \Omega $ is a set of positive reach; henceforth, by Theorem \ref{thm kleinjohann}, there exists $ t_0 > 0 $ such that $ \Sigma_t $ is a compact embedded $ \mathcal{C}^{1,1} $-hypersurface for every $ 0 < t < t_0 $; then, combining \eqref{prop: hk and n,h sets: eq13} and \eqref{umbilicality condition} we infer that $ \Sigma^\ast_t $ is a smooth umbilical embedded hypersurface relatively open in $ \Sigma_t $ for every $ t > 0 $, and $ \mathcal{H}^{n-1}(\Sigma_t \setminus \Sigma^\ast_t) =0 $ for $ \mathcal{L}^1 $ a.e.\ $ t > 0 $.
\end{proof}

We consider the upper-half space model of the hyperbolic space $ \mathbb{H}^n $: namely $ \mathbb{H}^n = \mathbb{R}^n_+ = \{x \in \mathbb{R}^n : x_n > 0\} $ and 
$$ g_{\mathbb{H}^n} = \frac{1}{x_n^2} (dx_1^2 + \ldots + dx_n^2). $$

\begin{proof}[Proof of Theorem \ref{thm main hyperbolic}]
	Choose $ p \in \mathbb{H}^n \setminus \overline{\Omega} $ and define $$ r(x) = \dist_{\mathbb{H}^n}(x,p), \quad  f(x) = \cosh(r(x)), $$
	for $ x \in \mathbb{H}^n $.  Recalling that the geodesic spheres in $ \mathbb{H}^n $ of radius $ \rho $ are smooth embedded umbilical hypersurface with all principal curvatures equal to $ \coth(\rho) $,  a straightforward computation gives that 
	$$ \Der^2 f =\sinh(r) \, \Der^2 r + \cosh(r) \, dr \otimes dr, $$
	\begin{equation}\label{thm main hyperbolic eq 1}
		\Der^2f(x)(v,v) = \cosh(r(x)),
	\end{equation}  
	\begin{equation}\label{thm main hyperbolic eq 2}
		\Delta f(x) = n\cosh(r(x)),
	\end{equation}
	\begin{flalign*}
		f(x)\, {\rm Ric}_{\mathbb{H}^n}(v,v) & - \Der^2 f(x)(v,v) + \Delta f(x)\\
		& = -(n-1)\cosh(r(x)) -\cosh(r(x)) + n \cosh(r(x)) =0,
	\end{flalign*} 
	for every $ v \in T_x(\mathbb{H}^n) $ with $|v|=1 $. Henceforth, the Riemannian manifold $ (\mathbb{H}^n \setminus \{p\}, g_{\mathbb{H}^n}) $ endowed with the function $ f $ satisfies the hypothesis of Theorem \ref{thm hk in substatic spaces}. Arguing as in the proof of Theorem \ref{thm main}, we notice that it is not restrictive to assume that $ \Omega $ is an open set such that $ \partial \overline{\Om} = \partial \Omega = \overline{\partial^\ast \Omega} $ and $ \mathcal{H}^{n-1}(\partial \Omega \setminus \partial^\ast \Omega) =0 $; by Allard regularity theorem we also have that $ \partial^\ast \Omega $ is a smooth embedded hypersurface. We set $ \Sigma = \partial^\ast \Omega $. By Theorem \ref{thm white}, $ \overline{\Sigma} $ is a White $(n-1, \lambda)$-set of $ \mathbb{H}^n $. Finally, we need to check that $ H_\Sigma $ is positive and the couple $(\Omega, \Sigma)$ fulfills the equality in \eqref{thm hk in substatic spaces inequality}. The equality \eqref{thm main hyperbolic eq} clearly implies that $ H_\Sigma(x) = \lambda $ for $ x \in \Sigma $. Moreover, since by \eqref{thm main hyperbolic eq 1}  and \eqref{thm main hyperbolic eq 2} we have that 
	\begin{equation*}
		{\rm div}^\Sigma (\nabla f)(x) = \Delta f(x) - \Der^2 f(\nu_\Omega(x), \nu_\Omega(x)) = (n-1)f(x) \quad \textrm{for $ x \in \Sigma $,}
	\end{equation*}
	we infer from \eqref{thm main hyperbolic eq} that
	\begin{flalign*}
		(n-1)\int_\Sigma f\, d\mathcal{H}^{n-1} = \int_{\Sigma}{\rm div}^\Sigma (\nabla f)\, d\mathcal{H}^{n-1} & = \lambda \int_\Sigma \langle \nu_\Omega, \nabla f \rangle_{\mathbb{H}^n}\, d\mathcal{H}^{n-1}\\
		& = \lambda \int_\Omega \Delta f\, d\mathcal{H}^n = n\lambda \int_\Omega f\, d\mathcal{H}^n.
	\end{flalign*}
	This implies that $ H_\Sigma \equiv \lambda > 0 $ and $(\Omega, \Sigma) $ fulfills the equality case in \eqref{thm hk in substatic spaces inequality}.
	We conclude that there exists $ t_0 > 0 $ such that the sets 
	$$ \Sigma_t = \{p \in  \Omega: \dist_{g^\ast}(p, \overline{\Sigma}) = t\}, \quad \textrm{for $ \mathcal{L}^1 $ a.e.\ $ 0 < t < t_0 $,} $$
	where $ g^\ast = \frac{g_{\mathbb{H}^n}}{f^2} $, are closed embedded $ \mathcal{C}^{1,1} $-hypersurfaces with respect to the hyperbolic metric. Since the hyperbolic metric is conformally equivalent to the Euclidean metric, and since changing conformally the metric of the ambient space preserves umbilicity,  we infer that $ \Sigma_t $ is also umbilical with respect to the Euclidean metric. Henceforth, by \cite[Lemma 3.2]{deRosaKolaSantilli}, we conclude that each $ \Sigma_t $ is a finite disjoint union of Euclidean spheres. Now the conclusion follows letting $ t \to 0 $.

\appendix

\section{Assumption ${\rm (H3)^*}$ and models in General Relativity}\label{appendix h3star} In this appendix we check that the Reissner--Nordstrom manifolds satisfy assumption ${\rm (H3)^*}$, while the deSitter--Schwarzschild manifolds do not. This simple fact, combined with the analysis of equality cases in Brendle's Heintze--Karcher type inequalities, shows that a stronger stability mechanism for almost-CMC hypersurfaces is at a play in the R--N manifolds. Indeed, when ${\rm (H3)^*}$ holds, Brendle's argument also provides, in addition to almost-umbilicality, a direct control on the oscillation of the normals with respect to the radial directions as measured by $g_N(\nu_\Om,\nu_\Om)$; see, for example, condition \eqref{for H3 star}.

\medskip

Let us set $G=S^{n-1}\times(s_1,s_2)$ for some $0<s_1<s_2\le+\infty$. The dS--S manifold is then $(G,g_{{\rm dSS}})$, where
\[
g_{{\rm dSS}}=\frac{ds\otimes ds}{1-m\,s^{2-n}-\kappa s^2}+s^2\,g_{S^{n-1}}\,,
\]
with
\[
m>0\,,\qquad -\infty<\kappa<(n-2)\,\Big(\frac{4}{n^n\,m^2}\Big)^{1/(n-2)}\,.
\]
When $\k>0$ the upper bound on $\k$ guarantees that $1-m\,s^{2-n}-\kappa s^2$ has exactly two zeros $s_1<s_2$ on $(0,\infty)$, while if $\k\le 0$ we set $s_2=+\infty$, while $s_1$ is the unique zero of $1-m\,s^{2-n}-\kappa s^2$ on $(0,\infty)$. The R--N manifold is defined instead as $(G,g_{{\rm RN}})$, where
\[
g_{{\rm RN}}=\frac{ds\otimes ds}{1-m\,s^{2-n}+q^2\,s^{4-2n}}+s^2\,g_{S^{n-1}}\,,\qquad m>2\,q>0\,.
\]
In this case $s_1$ is the largest of the two solutions of $1-m\,s^{2-n}+q^2\,s^{4-2n}=0$ on $(0,\infty)$, while we set $s_2=+\infty$. Both examples can be modeled as
\[
g_\om=(1/\om(s))\,ds\otimes ds +s^2\,g_{S^{n-1}}
\]
for a smooth function $\om:(s_1,s_2)\to(0,\infty)$. We then define
\[
F(s)=\int_{s_1}^s\frac1{\sqrt\om}\quad \forall s\in(s_1,s_2)\,,\qquad h(t)=F^{-1}(t)\quad\forall t\in(0,\bar{r})\,,\bar{r}:=F(s_2)\,,
\]
so that $h(F(s))=s$ for every $s\in(s_1,s_2)$, and
\begin{eqnarray}
  \label{diffff}
  h'(F(s))=\sqrt{\om(s)}\,,\qquad h''(F(s))=\om'(s)\big/2\,,\qquad\forall t\in(0,\bar{r})\,.
\end{eqnarray}
Setting $M=S^{n-1}\times(0,\bar{r})$, the map $\phi: M\to G$ defined by $\phi(\tau,t)=(\tau,h(t))$ is such that
\[
(d\phi)^*g_\om=dr\otimes dr+h(r)^2\,g_{S^{n-1}}=:g_h\,,
\]
so that $(G,g_\om)$ is isometric to $(M,g)$, and rigidity of CMC-hypersurfaces  in dS--S and R--N manifolds can be studied in their $(M,g)$ representations. Since ${\rm (H0)}$ holds with $\rho=1$ we have
\[
 2\,\frac{h''}h+(n-2)\,\frac{(h')^2-1}{h^2}\Big|_{F(s)}=\frac{\om'(s)}s+(n-2)\,\frac{\om(s)-1}{s^2}=\ell(s)\,,
\]
for every $s\in(s_1,s_2)$. We thus find, in the case of $g_{{\rm dSS}}$, where $\om(s)=1-m\,s^{2-n}-\k\,s^2$,
\[
\ell(s)=\frac1s\,\Big((n-2)\,m\,s^{1-n}-2\k\,s\Big)+\frac{n-2}{s^2}\,\Big(-m\,s^{2-n}-\k\,s^2\Big)=-n\,\k\,,
\]
so that ${\rm (H3)}$ holds (but ${\rm (H3)^*}$ does not); in the case of $g_{{\rm RN}}$, where $\om(s)=1-m\,s^{2-n}+q^2\,s^{4-2n}$, we have an identical cancellation of the mass term, but thanks to the $q^2$-term we rather find
\[
\ell(s)=-\frac{(n-2)\,q^2}{s^{2n-2}}\,,
\]
so that $\ell(s)$ is strictly increasing on $(s_1,s_2)$, and  ${\rm (H3)^*}$ holds.

\end{proof}

\end{document}